\documentclass[reqno]{amsart}
\usepackage{enumerate}

\makeatletter
\newtheorem*{rep@theorem}{\rep@title}
\newcommand{\newreptheorem}[2]{%
\newenvironment{rep#1}[1]{%
 \def\rep@title{#2 \ref{##1}}%
 \begin{rep@theorem}}%
 {\end{rep@theorem}}}
\makeatother

\newreptheorem{theorem}{Theorem}

\usepackage{amsthm,epsfig,amssymb,amsmath,mathrsfs,color}

\newcommand{\R}{\mathbb{R}}
\newcommand{\Rplus}{\mathbb{R}_+}
\newcommand{\N}{\mathbb{N}}

\newcommand{\union}{\cup}
\newcommand{\bigunion}{\bigcup}

\newcommand\intersection{\cap}
\newcommand\closure{\operatorname{cl}}
\newcommand\interior{\operatorname{int}}
\newcommand\image{\operatorname{Im}}
\newcommand\after{\circ}

\newcommand\conda{{(i)}}
\newcommand\condb{{(ii)}}
\newcommand\condc{{(iii)}}
\newcommand\condd{{(iv)}}

\newcommand\proj{{\rm P}}
\newcommand\measuredfoliations{\mathcal{MF}}
\newcommand\projmeasuredfoliations{\operatorname{\mathcal{{PMF}}}}
\newcommand\unitlams{\projmeasuredfoliations}

\newcommand\surface{S}
\newcommand\teichmullerspace{\mathcal{T}(S)}

\newcommand\extlength{\operatorname{Ext}}
\newcommand\unitfoliations{\mathcal{P}}

\newcommand\extfunc{\mathcal{E}}
\newcommand\dualextfunc{\mathcal{E}^*}

\newcommand\curves{\mathcal{S}}

\newcommand\teichdist{d}

\newcommand\length{L}
\newcommand\area{A}
\newcommand\dist{d}

\newcommand\geo{\gamma}

\newcommand\gentran{\tilde\mu}

\newtheorem*{theorem*}{Theorem}
\newtheorem*{proposition*}{Proposition}

\newtheorem{prop}{Proposition}[section]
\newtheorem{proposition}[prop]{Proposition}
\newtheorem{corollary}[prop]{Corollary}
\newtheorem{lemma}[prop]{Lemma}

\newtheorem{theorem}[prop]{Theorem}
\theoremstyle{definition}
\newtheorem{construction}[prop]{Construction}

\newtheorem{definition}[prop]{Definition}

\newenvironment{remark}[1][Remark.]{\begin{trivlist}
\item[\hskip \labelsep {\bfseries #1}]}{\end{trivlist}}
\begin{document}

\title{The asymptotic geometry of the Teichm\"uller metric}
\date{\today}
\author{Cormac Walsh}
\address{INRIA Saclay \& CMAP,
Ecole Polytechnique, 91128 Palaiseau, France}
\email{cormac.walsh@inria.fr}

%32G15 Moduli of Riemann surfaces, Teichmüller theory
%30F60 Teichmüller theory
%57M50 Geometric structures on low-dimensional manifolds


\begin{abstract}
We determine the asymptotic behaviour of extremal length along arbitrary
Teichm\"uller rays. This allows us to calculate the endpoint in the
Gardiner--Masur boundary of any Teichm\"uller ray. We give a proof that
this compactification is the same as the horofunction compactification.
An important subset of the latter is the set of Busemann points.
We show that the Busemann points are exactly the limits of the Teichm\"uller
rays, and we give a necessary and sufficient condition for a sequence
of Busemann points to converge to a Busemann point. Finally, we determine
the detour metric on the boundary.
\end{abstract}

\maketitle

\section{Introduction}
\label{sec:introduction}

Let $S$ be an oriented surface of genus $g$ with $n$ punctures. We assume that
$3g-3+n\ge 1$. The Teichm\"uller space $\teichmullerspace$ of $S$ may be
defined as the space of marked conformal structures on $S$ up to conformal
equivalence. Viewed in this way, the most natural metric of $\teichmullerspace$
is the Teichm\"uller metric. Kerckhoff~\cite{kerckhoff_asymptotic} has shown
that a useful tool for studying the geometry of the this metric is the
extremal length of a measured foliation. Here we examine the behaviour of the
extremal length as one travels along a Teichm\"uller geodesic ray.

Our main result is as follows.
Let $q$ be the initial quadratic differential of a geodesic ray, and let
$V(q)$ and $H(q)$ be, respectively, its vertical and horizontal measured
foliations.
Recall that removing the critical graph of a measured foliation $G$ decomposes
the surface
into a finite number of connected components, each of which is either
a cylinder of closed leaves or a minimal component in which all leaves are
dense. Furthermore, the transverse measure on a minimal component $D$
may be written as
a sum of a finite number of projectively-distinct ergodic measures:
$\nu|_D = \sum_j \nu_{D,j}$.
We say that a measured foliation $G'$ is an \emph{indecomposable component}
of $G$ if it is either one of the cylindrical components of $G$,
or it is topologically equivalent to one of the minimal components $D$ and has
as transverse measure one of the $\nu_{D,j}$.
A measured foliation is \emph{indecomposable} if it has only one
indecomposable component, namely itself.

\begin{theorem}
\label{thm:extreme_len_asymptotics}
Let $R(q;\cdot):\Rplus\to\teichmullerspace$ be the Teichm\"uller ray with
initial unit-area quadratic differential $q$, and let $F$ be a measured
foliation. Then,
\begin{align*}
\lim_{t\to\infty} e^{-2t} \extlength_{R(q;t)}[F]
   = \sum_j \frac{i(G_j,F)^2}{i(G_j,H(q))},
\end{align*}
where the $\{G_j\}$ are the indecomposable components of the vertical foliation
$V(q)$.
\end{theorem}
This result was proved by Kerckhoff~\cite{kerckhoff_asymptotic}
in the case of Jenkins--Strebel rays, that is, when all the indecomposable
components $\{G_j\}$ are annular. See also~\cite{ivanov_isometries},
\cite{miyachi_teichmuller}, and~\cite{farb_masur_deligne} for a more explicit
treatment.

In~\cite{gardiner_masur_extremal}, Gardiner and Masur introduced a
compactification of Teichm\"uller space by embedding it into the projective
space of $\R^S$ using the (square root of) the extremal length function,
and showing that the image is relatively compact. This is analogous
to the Thurston compactification, the only difference being that
extremal lengths are used rather than hyperbolic lengths.

Theorem~\ref{thm:extreme_len_asymptotics} is exactly what is needed to show
that Teichm\"uller rays converge in the Gardiner--Masur compactification,
and to calculate their limits.

\begin{corollary}
\label{cor:busemann_formula}
The Teichm\"uller ray $R(q;\cdot)$ converges in the Gardiner--Masur
compactification to the projective class of
\begin{align*}
\extfunc_q(\cdot)
   := \Big(\sum_j \frac{i(G_j,\cdot)^2}{i(G_j,H(q))} \Big)^{1/2}.
\end{align*}
\end{corollary}

The Jenkins--Strebel case of this corollary appears
in~\cite{miyachi_gardiner} and~\cite{miyachi_teichmuller}.

There exists a very general means of compactifying a metric space, namely the
horofunction compactification,
introduced by Gromov~\cite{gromov:hyperbolicmanifolds}.
In~\cite{walsh_stretch}, it was shown that the horofunction compactification
of Teichm\"uller space with
\emph{Thurston's Lipschitz metric}~\cite{thurston_minimal} is in fact just
the usual Thurston compactification.
Recall that this metric is the logarithm of the least possible Lipschitz
constant over all diffeomorphisms of the surface isotopic to the identity,
and that a formula similar to Kerckhoff's holds, but with extremal length
replaced by hyperbolic length.

We show in Section~\ref{sec:horoboundary_is_GM} that the horofunction
compactification of Teichm\"uller's metric is the same as the
Gardiner--Masur compactification. This was previously proved by
Liu and Su~\cite{liu_su_compactification}.

\begin{theorem}
\label{thm:homeo}
The Gardiner--Masur compactification and the horofunction compactification
of $\teichmullerspace$ are the same.
\end{theorem}

It seems that it may be more appropriate to consider the Gardiner--Masur
compactification when one takes the conformal view of Teichm\"uller space,
and the Thurston compactification when one takes the hyperbolic view.
Results about the convergence of Teichm\"uller geodesics to points
in the Thurston boundary, such as
in~\cite{lenzhen_geodesics} and~\cite{masur_two_boundaries}, may be seen as
attempts to relate these two geometries of Teichm\"uller space.

A particularly interesting subset of the horofunction boundary is its set
of \emph{Busemann points}. These are the boundary points that can be reached
as a limit along an \emph{almost-geodesic}, which is a slight weakening
of the usual notion of geodesic. Because Teichm\"uller rays are geodesic
in the Teichm\"uller metric, it is clear that the horofunctions corresponding
to the points of the Gardiner--Masur boundary identified in
Corollary~\ref{cor:busemann_formula} are Busemann.
We show that these are the only ones.
\begin{theorem}
A horofunction is a Busemann point if and only if it corresponds to a point
of the form $\extfunc_q$ in the Gardiner--Masur boundary, where $q$ is a
quadratic differential.
\end{theorem}

It is also of interest to know which Teichm\"uller rays converge to the
same boundary point.
Kerckhoff answered a related question in the case of Jenkins--Strebel rays
using the notion of \emph{modular equivalence} of Jenkins--Strebel quadratic
differentials.
He showed that two Jenkins--Strebel rays are asymptotic, that is,
the distance between them converges to zero, if and only if their initial
quadratic differentials are modularly equivalent. We generalise this notion
to arbitrary quadratic differentials.

\begin{definition}
Let $q$ and $q'$ be two quadratic differentials whose vertical foliations
can be simultaneously be written in the form
$V(q) = \sum_j \alpha_j G_j$ and $V(q') = \sum_j \alpha'_j G_j$,
where $\{G_j\}_j$ is a set of mutually non-intersecting indecomposable
measured foliations and $\{\alpha_j\}_j$ and $\{\alpha'_j\}_j$ are sets
of positive coefficients.
We say that $q$ and $q'$ are \emph{modularly equivalent} if
\begin{align}
\label{eqn:ratios}
\frac{\alpha_j}{i(G_j, H(q))}
   = C \frac{{\alpha'_j}}{i(G_j, H(q'))},
\qquad\text{for all $j$},
\end{align}
where $C$ is a positive constant independent of $j$.
\end{definition}

\begin{theorem}
\label{thm:modular_bijection}
Two Busemann points $\extfunc_{q}$ and $\extfunc_{q'}$ are identical
if and only if $q$ and $q'$ are modularly equivalent.
\end{theorem}

The Jenkins--Strebel case of the following result appears in~\cite{jenkins}
and~\cite{strebel}.

\begin{theorem}
\label{thm:modular}
Every modular equivalence class of quadratic differentials has a representative
at each point of Teichm\"uller space. This representative is unique up to
multiplication by a positive constant.
\end{theorem}

The above theorems have the following geometric interpretation.

\begin{theorem}
\label{thm:busemann_bijection}
Let $p$ be a point of $\teichmullerspace$ and $\xi$ be a Busemann point
of the horofunction boundary. Then, there exists a unique geodesic ray
starting at $p$ and converging to $\xi$.
\end{theorem}

The uniqueness part if this theorem was proved independently by
Miyachi~\cite{miyachi_teichmuller2}.

We have seen that the set of Busemann points may be identified with the
set of unit-area quadratic differentials at the basepoint.
This is also the case for the Teichm\"uller boundary of Teichm\"uller space,
so it is interesting to compare the two. Let $x_n$ be a sequence in
$\teichmullerspace$, and write $x_n=R(q_n;t_n)$, where $t_n$ is the
distance to the basepoint and $q_n$ is the initial quadratic differential
at the basepoint. Recall that $x_n$ converges to a point in the
Teichm\"uller boundary if and only if $q_n$ converges to a unit area quadratic
differential $q$, and $t_n$ converges to infinity.
An equivalent condition is that the geodesic segment connecting the basepoint
to $x_n$ converges uniformly on compact sets of $\Rplus$ to $R(q;\cdot)$.

Let $q$ and $q'$ be Jenkins--Strebel differentials.
Kerckhoff~\cite{kerckhoff_asymptotic} showed that if $V(q)$ and $V(q')$
have the same single core curve, then $R(q;\cdot)$ and $R(q';\cdot)$
have the same limit in the Teichm\"uller boundary. He also showed that
the same result is true when the vertical foliations have the same $3g-3$
core curves, which is the maximum number possible.
Concerning the general case, he says ``it seems likely that some
non-convergent rays exist''. We show that, in fact, no such rays exist.

\begin{theorem}
\label{thm:convergent_rays}
Each Teichm\"uller ray converges in the Teichm\"uller compactification.
Two rays have the same limit if and only if their initial quadratic
differentials are modularly equivalent.
\end{theorem}

We describe the topology that the set of Busemann points inherits from the
horofunction boundary in Theorem~\ref{thm:convergence_criterion}.
It turns out to be strictly finer than the topology on the Teichm\"uller
boundary. This implies in particular that there exist non-Busemann points
in the horofunction boundary when $3g-3+n\ge 2$,
a result that has also been proved by Miyachi~\cite{miyachi_non-busemann}.

Our final result concerns the detour metric on the set of Busemann points
of Teichm\"uller space. An explicit formula for this may be found in
Corollary~\ref{cor:detour_metric}, where it is seen that the distance
between two Busemann points $\extfunc_{q}$
and $\extfunc_{q'}$ is finite if and only if their vertical foliations
can be simultaneously be written $V(q) = \sum_j \alpha_j G_j$
and $V(q') = \sum_j \alpha'_j G_j$,
where $\{G_j\}_j$ is a set of mutually non-intersecting indecomposable
measured foliations and $\{\alpha_j\}_j$ and $\{\alpha'_j\}_j$ are sets
of positive coefficients.
It is interesting that this is exactly the criterion for when the two
Teichm\"uller rays $R(q;\cdot)$ and $R(q';\cdot)$ stay a bounded distance
apart---the various cases are considered in~\cite{masur_class, masur_uniquely,
ivanov_isometries, lenzhen_masur}.
Of course, one may easily show in general that when rays stay a bounded
distance apart, the detour metric between the corresponding Busemann points
is finite. However, the converse is not true for general metric spaces.

The layout of this paper is as follows. In Section~\ref{sec:background},
we recall some background material on Teichm\"uller space, including its
Gardiner--Masur compactification. In Section~\ref{sec:lower_bound}, we prepare
to prove Theorem~\ref{thm:extreme_len_asymptotics} by calculating a lower bound
on the extremal length. The upper bound completing the proof is established in
Section~\ref{sec:upper_bound}, which is considerably longer.
We recall the basics about the horofunction compactification in
Section~\ref{sec:horofunction_boundary}, and prove Theorem~\ref{thm:homeo}
in Section~\ref{sec:horoboundary_is_GM}. Section~\ref{sec:busemann} is devoted
to modular equivalence and the various convergence results detailed above.
Finally, in Section~\ref{sec:detour}, we calculate the detour cost on the
boundary.

\section{Background}
\label{sec:background}
Let $S$ be an oriented surface of genus $g$ with $n$ punctures.
We assume that $S$ has negative Euler characteristic and is not the
$3$-punctured sphere, in other words, that $3g+3-n \ge 1$.
The Teichm\"uller space $\teichmullerspace$ of $S$ is the space of marked
conformal structures $(X,f)$ on $S$ up to conformal equivalence.
Here $X$ is a surface and $f:S\to X$ is a quasi-conformal map.
Recall that two marked conformal structures $(X_1,f_1)$ and $(X_1,f_1)$
are conformally equivalent if there exists a conformal map $f:X_1\to X_2$
such that $f\after f_1$ is homotopic to $f_2$.

Let $x_1:=(X_1,f_1)$ and $x_2:=(X_2,f_2)$ be two marked conformal structures
on $S$. The Teichm\"uller distance between $x_1$ and $x_2$ is defined
to be
\begin{align*}
d(x_1,x_2) := \frac{1}{2} \log \inf_f K(f),
\end{align*}
where the infimum is over all quasi-conformal homeomorphisms $f:X_1\to X_2$
that are homotopic to $f_2\after f_1^{-1}$, and $K(f)$ is the quasi-conformal
dilatation of $f$. Obviously, $d(x_1,x_2)$ remains the same if $x_1$ or $x_2$
are replaced by a conformally equivalent structure, and so $d$ defines a
metric on $\teichmullerspace$, called the \emph{Teichm\"uller metric}.
This metric is complete and geodesic~\cite{kravetz}.

A (holomorphic) quadratic differential on a Riemann surface $X$ is a tensor
of the form $q(z)\mathrm{d}z^2$, where $q$ is holomorphic.
Quadratic differentials are allowed to have $1$st order poles at the punctures.

A quadratic differential has a finite number of zeros. In a neighbourhood
of any other point, there is a natural parameter $z=x+ iy$.
Thus, a quadratic differential $q$ gives rise to two measured foliations
on $S$: the horizontal foliation $H(q)$ and the vertical foliation $V(q)$.
The leaves of $H(q)$ are defined by $y=\textrm{constant}$, and the transverse
measure is $|\mathrm{d}x|$. Similarly, the leaves of $V(q)$ are defined by
$x=\textrm{constant}$, and the transverse measure is $|\mathrm{d}y|$.
The foliations $H(q)$ and $V(q)$ each have a singularity at every zero of $q$.
At a zero of order $k$, the singularity is $(k+2)$-pronged.
See~\cite{fathi_laudenbach_poenaru} for a detailed account of measured
foliations.

We always consider there to be singularities at the punctures. They may be
one-pronged, two-pronged, or higher order.

The metric $\mathrm{d}x^2 + \mathrm{d}y^2$ is the singular flat metric
determined by $q$. Its total area is finite. A quadratic differential is
said to be of unit area if the area of its associated flat singular
metric is $1$.

At each point $x=(X,f)$ in Teichm\"uller space, there is a one-to-one
correspondence between the set of geodesic rays starting at $x$ and the set
of unit-area quadratic differentials on $X$. Given such a quadratic
differential $q$ on $X$ and a scalar $K>0$, one multiplies the transverse
measure of $V(q)$ by $K$ and the transverse measure of $H(q)$ by $1/K$.
The resulting pair of measured foliations determines a conformal structure
on $S$, and hence a point in Teichm\"uller space. We denote this point by
$R(q;t)$, where $t=\log K$. The function $t\mapsto R(q;t)$ is a
unit-speed geodesic ray.

We say that a leaf of a foliation is a \emph{saddle connection} if it joins
two not necessarily distinct singularities. The \emph{critical graph}
is the union of all saddle connections.
The complement of the critical graph has a finite number of connected
components. Each is either a cylinder swept out by closed leaves
or a so-called \emph{minimal} component, in which all leaves are dense.
On each minimal component $D$, there exists a finite set of ergodic transverse
measures $\mu_1,\dots,\mu_n$ such that any transverse measure
$\mu$ on $D$ can be written as a sum $\mu=\sum_{i=1}^n f_i \mu_i$,
with non-negative coefficients $\{f_i\}$.
There is an upper bound on the number $n$ of ergodic transverse measures
that just depends on the topology of the surface.

A conformal metric on a Riemann surface is a metric that is locally of the
form $\rho(z)|\mathrm{d}z|$, where $\rho$ is a non-negative measurable
real-valued function on the surface.
Let $\curves$ be the set of free homotopy classes of essential, non-peripheral
simple closed curves of $S$.
We define the $\rho\,$-length of a curve class $\alpha\in\curves$
to be the length of the shortest curve in $\alpha$
measured with respect to $\rho$, that is,
\begin{align*}
\length_\rho(\alpha) :=
   \inf_{\alpha'\in[\alpha]}\int_{\alpha'} \rho \, |\mathrm{d}z|,
\end{align*}
where $|\mathrm{d}z|$ denotes the Euclidean length element.
The area of $\rho$ is defined to be
$\area(\rho):= \int_S \rho^2 \, \mathrm{d}x\, \mathrm{d}y$.

The \emph{extremal length} of a curve class $\alpha\in\curves$ on a
Riemann surface $X\in\teichmullerspace$ is
\begin{align*}
\extlength_X(\alpha) := \sup_\rho \frac{\length_\rho(\alpha)^2}{\area(\rho)},
\end{align*}
where the supremum is over all Borel-measurable conformal metrics of finite
area.
This is the so-called analytic definition of extremal length.
There is also the following geometric definition:
\begin{align*}
\extlength_X(\alpha) := \inf_C\frac{1}{\operatorname{mod}(C)}.
\end{align*}
Here the infimum is over all embedded cylinders $C$ in $X$ with core curve
isotopic to $\alpha$, and $\operatorname{mod}(C)$ is the modulus of $C$.

\begin{lemma}[Gardiner--Masur~\cite{gardiner_masur_extremal}]
\label{lem:inversion_formula}
For all measured foliations $F$ and $G$, and points $x\in\teichmullerspace$,
\begin{align*}
\extlength_x(G) = \sup_{F\in\unitfoliations}
                    \frac{i(G,F)^2}{\extlength_x(F)}.
\end{align*}
\end{lemma}

\subsection{The Gardiner--Masur boundary}

Define the map $\Phi: \teichmullerspace \to P\R^S$, so that
$\Phi(x)$ is the projective class of
$(\extlength_x(\alpha)^{1/2})_{\alpha\in\curves}$.
In~\cite{gardiner_masur_extremal}, Gardiner and Masur showed that
$(\Phi,\closure\image\Phi)$ is a compactification of Teichm\"uller space.
Following~\cite{miyachi_teichmuller}, we call this the
\emph{Gardiner--Masur compactification},
and its topological boundary the \emph{Gardiner--Masur boundary}.

\section{Lower bound}
\label{sec:lower_bound}

We use $\Rplus$ to denote the set of non-negative real numbers.
Recall that we have defined, for any quadratic differential $q$, the function
$\extfunc_q:\measuredfoliations\to\Rplus$,
\begin{align*}
\extfunc_q(\cdot):=\Big(
   \sum_j \frac{i(G_j,\cdot)^2}{i(G_j,H(q))}
   \Big)^{1/2},
\end{align*}
where the $\{G_j\}_j$ are the indecomposable components of $V(q)=\sum_j G_j$.

\begin{lemma}
\label{lem:optimise}
Let $a$ and $b$ be vectors in $\Rplus^n\backslash\{0\}$, $n\ge 1$,
and assume that there is no coordinate $j$ for which both $a_j$ and $b_j$ are
zero. Then, the function from $\Rplus^n\backslash\{0\}$ to $\R$ defined by
\begin{align*}
x \mapsto \frac{(\sum_{j}a_j x_j)^2}{\sum_j b_j x_j^2}
\end{align*}
attains its supremum when
%$\{(a_1\lambda/b_1,\dots,a_n\lambda/b_n) \mid \lambda > 0\}$.
$x_j = C a_j / b_j$, where $C>0$ is a constant independent of~$j$.
%$x_j b_j / a_j = \text{const.}$ for all $j$.
The supremum is $\sum_j a_j^2 / b_j$.
\end{lemma}
\begin{proof}
This is elementary.
\end{proof}

\begin{lemma}
\label{lem:finite_time_bound}
Let $q$ be a quadratic differential, and let $R(q;\cdot)$ be the
associated geodesic ray. Then,
\begin{align*}
e^{-2t}{\extlength_{R(q;t)}(F)} \ge \extfunc_q^2(F),
\end{align*}
for all $t\in\Rplus$ and $F\in\measuredfoliations$.
\end{lemma}

\begin{proof}
Fix $t\in\Rplus$ and let $\alpha\in\curves$.
Decompose the vertical foliation of $q$ into its indecomposable components:
$V(q)=\sum_{j=0}^J G_j$.

Define a conformal metric $\rho:\surface\to\Rplus$ as follows.

On the annulus associated to each annular indecomposable 
component $G_j$, let $\rho$ take some positive value $\rho_j$,
which we will choose later.

Let $D$ be a minimal domain of $V(q)$, and take a horizontal arc $I$ in the
interior of $D$. By considering the point of first return of leaves starting
on $I$, we obtain a (non-oriented) interval exchange map,
and hence a decomposition of $D$ into a finite number of rectangles $\{R_l\}_l$.
There is a one-to-one correspondence between the ergodic measures of the
interval exchange map and the indecomposable measured foliations
that are supported on $D$. Consider the subset of these indecomposable measured
foliations that appear as indecomposable components of $V(q)$.
Denote this subset by $\{G_j\}$; $j\in J_D$, where $J_D\subset J$.
Write $G_j = (G,\nu_j)$ for all $j\in J_D$, where $G$ is the unmeasured
foliation obtained from $V(q)$ by forgetting the measure.

Consider one of the rectangles $R_l$.
We can write $R_l = X\times Y$, where $X$ is a horizontal arc and $Y$ is
a vertical arc.
Since the transverse measures $\{\nu_j\};{j\in J_D}$ are mutually singular,
there exists a decomposition $X=\bigunion_{j\in J_D}X_j$ of $X$ into disjoint
Borel subsets $\{X_j\};{j\in J_D}$ such that $\nu_j[X_k]$ equals $\nu_j[X]$
when $j=k$, and is zero otherwise. Define $\rho$ to take some positive value
$\rho_j$ on $X_j \times Y$, for each $j\in J_D$.

Do this for every rectangle $R_l$ and every minimal domain $D$.
Note that points on horizontal edges belong to more than one rectangle,
and hence the value of $\rho$ has been defined more than once on these points.
This is not a problem however since  the set of such points where the
definitions differ has $V(q)$-measure zero.

The value of $\rho$ on vertical edges and on the critical graph is not
important for the present argument.

For any simple closed curve $\alpha$,
\begin{align*}
\int_\alpha \rho \,|\mathrm{d}z|
   &\ge e^{t} \int_\alpha \rho \,\mathrm{d}V(q) \\
   &= e^{t} \sum_{j=0}^J \rho_j \int_\alpha \mathrm{d}G_j,
\end{align*}
since $\rho$ is $G_j$-almost everywhere constant along $\alpha$, for all $j$.
But $\int_\alpha \mathrm{d}G_j\ge i(G_j,\alpha)$, and so
\begin{align*}
L_\rho(\alpha) \ge e^{t} \sum_{j=0}^J \rho_j i(G_j,\alpha).
\end{align*}
The area of $\rho$ is independent of $t$:
\begin{align*}
A(\rho) = \sum_{j=0}^J \rho_j^2 i(G_j,H(q)).
\end{align*}
Therefore,
\begin{align*}
e^{-2t}{\extlength_{R(q;t)}(\alpha)}
   \ge \frac{\Big(\sum_{j=0}^J \rho_j i(G_j,\alpha)\Big)^2}
            {\sum_{j=0}^J \rho_j^2 i(G_j,H(q))}.
\end{align*}
According to Lemma~\ref{lem:optimise}, the expression on the right-hand-side
attains its maximum when
$\rho_j=C i(G_j,\alpha)/i(G_j,H(q))$ for all $j$, where $C$ is any positive
constant. Moreover, its maximum is $\sum_{j=0}^J i(G_j,\alpha)^2/i(G_j,H(q))$.

This proves the theorem in the case where $F$ is a curve class.
The general case now follows using the continuity and homogeneity of extremal
length.
\end{proof}

\section{Upper bound}
\label{sec:upper_bound}

\newcommand\curveof[1]{a({#1})}
\newcommand\piecewisestraights{A}
\newcommand\probmeasures{M}
\newcommand\maptoS{r}
\newcommand\const{D}
\newcommand\asscurve{a}
\newcommand\dee{\mathrm{d}}
\newcommand\completionS{\bar S}
\newcommand\unionrects{\Gamma}
\newcommand\unordereds{\Omega}

As Kerchkoff observed in~\cite{kerckhoff_asymptotic}, one often uses
the analytic definition to obtain a lower bound on the extremal length, and
the geometric definition to obtain an upper bound. However,
we will not use this technique in this paper. Instead, we will
use the analytic definition a second time to establish another lower bound
with a different scaling, and then convert it into an upper bound using
Lemma~\ref{lem:inversion_formula}.

Let $\completionS$ denote the completion of $\surface$. The punctures of
$\surface$ are considered to be distinguished points of $\completionS$.

We define a rectangulation of a quadratic differential $q$ on $\surface$
to be a map $\maptoS$ from a disjoint union of a finite number $n$ of
rectangles $\unionrects := \sqcup_{k=1}^n [0,X_k]\times[0,Y_k]$
to $\completionS$ satisfying the following conditions:
\begin{enumerate}
\item
$\maptoS$ is surjective and continuous;
\item
$\{x\}\times(0,Y_k)$ is mapped into a leaf of $V(q)$, and $(0,X_k)\times\{y\}$
is mapped into a leaf of $H(q)$, for all $k$, and $x\in(0,X_k)$
and $y\in(0,Y_k)$;
\item
$\maptoS$ restricted to the union of the interiors of the rectangles is
injective, and the image is in $\surface$;
\item
$\maptoS$ restricted to the interior of any rectangle is an isometry,
using the Euclidean metric on $\unionrects$ and the singular flat metric
associated to $q$ on $S$.
\end{enumerate}
Denote by $\unordereds$ the set of unordered pairs
$(p,q)\in\unionrects\times\unionrects$
such that $p$ and $q$ lie in the boundary of the same rectangle.
We take on this set its natural topology coming from the product topology
on $\unionrects$.
For $(p,q)\in\unordereds$, we denote by $[p,q]$ the closed line segment
between $p$ and $q$ in the rectangle in which they both lie.
The expressions $[p,q)$, $(p,q)$, and $(p,q]$ will have their obvious meanings.

Let $\probmeasures$ be the space of Borel measures on
$\unordereds$.
For any measure $\mu$, let $|\mu|$ denote its total mass.

We say a point of $\surface$ is a \emph{corner} point if it is the image under
$\maptoS$ of a corner of a rectangle.
A \emph{weighting} $\rho$ of a rectangulation is an assignment of a positive
real number $\rho_k$ to each rectangle.

Define on $\unordereds$ the functions
\begin{align*}
v(p,q) &:=\int_{\maptoS[p,q]}\dee H(q)
\qquad\text{and} \\
h(p,q) &:=\int_{\maptoS[p,q]}\dee V(q).
\end{align*}
The length of $\maptoS[p,q]$
in the singular flat metric associated to $q$ is then
\begin{align*}
||(p,q)||:=(v(p,q)^2 + h(p,q)^2)^{1/2}.
\end{align*}
Let
\begin{align*}
l := \min\{||(p,q)|| \mid \textrm{$(p,q)\in\unordereds$,
$\maptoS(p)$ and $\maptoS(q)$ are distinct corner points}\}.
\end{align*}

Let $\piecewisestraights$ be the set of elements $\mu$ of $\probmeasures$
consisting of a finite number of atoms of mass $1$ on pairs $(p_n,q_n)$ such
that
\begin{enumerate}[(i)]%for small alpha-characters within brackets.
\item
\label{item_i}
after reordering if necessary, $\maptoS(p_{n+1})=\maptoS(q_n)$ for all $n$,
cyclically;
\item
\label{item_ii}
if, for any $n$, neither $[p_n,q_n]$ nor $[p_{n+1},q_{n+1}]$ are horizontal,
then $\maptoS[p_n,q_n]$ concatenated with
$\maptoS[p_{n+1},q_{n+1}]$ is an arc transverse to the horizontal foliation
$H(q)$;
\item
\label{item_iii}
if there is an atom on $(p,q)$ with both $p$ and $q$ lying in the same
horizontal edge of a rectangle,
then both $\maptoS(p)$ and $\maptoS(q)$ are corner points of $\surface$;
\item
\label{item_iv}
if for any $n$, we have $||(p_n,q_n)||<l$, then either $(p_{n+1},q_{n+1})$
is horizontal, or $(p_{n-1},q_{n-1})$ is.
\end{enumerate}
Note that each element $\mu$ of $\piecewisestraights$ defines a closed curve
$\curveof{\mu}$ on $\surface$, although this curve is not necessarily simple.
When considering an element of $\piecewisestraights$, we always reorder
the atoms in such a way that $(\ref{item_i})$ holds, and treat the index as
being cyclical.

\begin{lemma}
\label{lem:straighten}
Assume a rectangulation $r:\unionrects\to\completionS$ and a weighting $\rho$
is given, and let $\epsilon>0$.
Then, for every simple closed curve
$\alpha\in\curves$ there exists $\mu\in\piecewisestraights$
such that $\asscurve(\mu)$ is homotopic to $\alpha$, and
\begin{subequations}
\label{eqn:not_too_long}
\begin{align}
\int_{\asscurve(\mu)} \rho\,\dee H(q)
   &\le \int_{\alpha} \rho\,\dee H(q) + \epsilon,
\qquad\textrm{and} \\
\int_{\asscurve(\mu)} \dee V(q)
   &\le \int_{\alpha} \dee V(q) + \epsilon.
\end{align}
\end{subequations}
\end{lemma}
\begin{proof}
By perturbing $\alpha$ if necessary, we may assume that it passes only
finitely many times through the image under $\maptoS$ of the boundaries of
the rectangles.

Construct a measure $\mu$ on $\unordereds$ as follows.
For each piece of $\alpha$ lying in
the image of a rectangle $R$ and having endpoints $\maptoS(p)$ and $\maptoS(q)$
with $p$ and $q$ in $R$, put an atom of mass one on $(p,q)$. Clearly,
$\mu$ satisfies $\conda$, and $\asscurve(\mu)$ is homotopic to $\alpha$.
Moreover, $(\ref{eqn:not_too_long})$ holds if the perturbation is small enough.

So, order the pairs as in $\conda$.
Suppose that, for some $n$, the arc $\maptoS[p_n,q_n]$ concatenated
with $\maptoS[p_{n+1},q_{n+1}]$ is not transverse to the horizontal foliation,
and that neither $[p_n,q_n]$ nor $[p_{n+1},q_{n+1}]$ is horizontal.
So there exists a leaf segment with one end on $\maptoS[p_n,q_n]$ and the
other on $\maptoS[p_{n+1},q_{n+1}]$ that forms a disk when concatenated
with a subarc of $\maptoS[p_n,q_n]\cdot\maptoS[p_{n+1},q_{n+1}]$.
Choose the leaf segment in such a way as to maximise the size of the disk.

If $(p_n,q_n)$ and $(p_{n+1},q_{n+1})$ lie in the same rectangle, then simply
remove these two atoms from $\mu$, and replace them with an atom on
$(p_n,q_{n+1})$.

Otherwise, there is a point $x$ lying on the leaf segment, and elements $q'$
and $p'$ of $\unionrects$ lying in vertical edges of the rectangles containing,
respectively, $q_n$ and $p_{n+1}$, such that $x=\maptoS(q')=\maptoS(p')$.
Since the leaf segment was chosen to maximize the size of the disk, either
$\maptoS(q')=\maptoS(p')$ is a singular point, or one or both of $(p_n,q')$
or $(p',q_{n+1})$ is horizontal. So, in $\mu$, replace the atoms on
$(p_n,q_n)$ and $(p_{n+1},q_{n+1})$ with atoms on $(p_n,q')$ and $(p',q_{n+1})$.

Note that this replacement does not increase the number of atoms in $\mu$.

Repeating the procedure if necessary, we obtain an element $\mu$ on $M$
satisfying $\conda$, $\condb$, and~(\ref{eqn:not_too_long}),
such that $\asscurve(\mu)$ is homotopic to $\alpha$.

Now suppose there is an atom $(p_n,q_n)$ in $\mu$ with both $p_n$ and $q_n$
lying
in the same horizontal edge of a rectangle and $\maptoS(p_n)$ is not a corner
point. Consider the points along the straight line segment from $p_n$ to $q_n$
that are mapped by $\maptoS$ to corner points. Let $p'$ be the closest one
to $p$ if one exists; otherwise, let $p':= q_n$ .
Since $\mu$ satisfies $\conda$, we have $\maptoS(q_{n-1})=\maptoS(p_n)$.
None of the points between $p_n$ and $p'$ are mapped to corner points,
and so there is a point $q'$ in the same rectangle as $q_{n-1}$ such that
$\maptoS(q')=\maptoS(p')$. See Figure~\ref{figure1}.

\begin{figure}
\label{figure1}
\caption{figure1}
\end{figure}

Replace the atoms on $(p_{n-1},q_{n-1})$ and $(p_n,q_n)$ with atoms on
$(p_{n-1},q')$ and $(p',q)$ if $p'\neq q_n$, or just with $(p_{n-1},q')$
if $p'=q_n$. In the former case, condition~{\condb} is preserved.
In the latter case, this condition may not be preserved, so we must go back
to the previous step to re-establish it. Note, however, that in this case
the number of atoms in $\mu$ is decreased. This ensures that our construction
will eventually terminate.

By repeating this process as many times as necessary, we ensure that our
measure $\mu$ satisfies~{\conda}, {\condb}, {\condc},
and~(\ref{eqn:not_too_long}),
and that $\asscurve(\mu)$ is homotopic to $\alpha$.

Now suppose that there is an atom $(p_n,q_n)$ in $\mu$ satisfying
$||(p_n,q_n)||<l$. If $h(p_n,q_n)>0$ then either $p_n$ or $q_n$ lies in the
interior of a horizontal edge and the other point lies on a vertical edge.
Without loss of generality, assume the former case.
We can move $p_n$ without increasing $\int_{\curveof{\mu}}\dee V(q)$ until
$p_n$ coincides with a corner of the rectangle in which it lies.
If $p_n$ now equals $q_n$, we remove this atom and return to re-establish
\condb.
If they differ, we have now reduced to the case where $h(p_n,q_n)=0$.

So, consider the case where $h(p_n,q_n)=0$.
If $(p_{n+1},q_{n+1})$ is horizontal, then we have established the conclusion
of {\condc}. If not, then $\maptoS[p_n,q_n]$ concatenated with
$\maptoS[p_{n+1},q_{n+1}]$ is transverse to $H(q)$. Let $R_j$ and $R_k$
be the rectangles containing $(p_n,q_n)$ and $(p_{n+1},q_{n+1})$,
respectively. If $\rho_j<\rho_k$, then we can move $q_n$ towards $p_n$
without increasing $\int_{\curveof{\mu}}\rho\,\dee H(q)$, until either
$\maptoS(q_n)$ is a corner point or $q_n$ equals $p_n$. Similarly,
if $\rho_j>\rho_k$, then we can move $q_n$ away from $p_n$ without
increasing $\int_{\curveof{\mu}}\rho\,\dee H(q)$, until either
$\maptoS(q_n)$ is a corner point or $(p_{n+1},q_{n+1})$ is horizontal.
In the same way, we can move $p_n$ until either $\maptoS(p_n)$ is a corner
point, $(p_{n-1},q_{n-1})$ is horizontal, or $p_n$ and $q_n$ coincide.
If $p_n$ and $q_n$ now coincide, we may remove this atom from $\mu$
and then go back to re-establish \condb. If $\maptoS(p_n)$ and $\maptoS(q_n)$
are both corner points, then $||(p_n,q_n)||\ge l$ and {\condc} no longer
applies. If $(p_{n-1},q_{n-1})$ or $(p_{n+1},q_{n+1})$ is horizontal, then
the conclusion of {\condc} is true.
\end{proof}

Let $P$ be the subset of $\unordereds$ consisting of points of the form
$(p,p)$.

\begin{lemma}
\label{lem:no_degenerate_atoms}
Let $\mu_n$ be a sequence in $\piecewisestraights$, and let $\lambda_n$ be a
sequence of positive real numbers such that $\lambda_n\mu_n$
converges to $\mu\in\probmeasures$. Then, $\mu[P]\le 2\mu[H]$.
\end{lemma}
\begin{proof}
Observe that the set
\begin{align*}
P^l := \{ (p,q)\in\unordereds \mid ||(p,q)||<l \}
\end{align*}
is open, and the set $H$ is closed. Also, by {\condc} and {\condd},
$\mu_n[P^l]\le 2\mu_n[H]$, for all~$n$.
The conclusion now follows since $P\subset P^l$.
\end{proof}

\newcommand\patch{\mathcal{X}}
\newcommand\borels{B}
\newcommand\boundary{\partial}
\newcommand\rects{\mathcal{R}}

For any subset $\rects$ of the set of rectangles, define
\begin{align*}
\patch_\rects:=  \Big\{ x\in\surface \mid
    \maptoS^{-1}(x)\subset\union_{R\in\rects} R \Big\}.
\end{align*}
We call $\rects$ a \emph{patch} if $\union_{R\in\rects}\maptoS(R)$
is connected and simply connected, and $\patch_\rects$ contains no
singularities.
We say an arc $\alpha$ in $S$ is \emph{short} if $\maptoS^{-1}(\alpha)$
is contained within $\patch_\rects$ for some patch $\rects$.

Given a patch $\rects$, we may choose in a consistent way one of the horizontal
edges of each rectangle $R$ in $\rects$ to be the ``upper'' edge. By consistent,
we mean that if a vertical leaf segment is common to the image under $\maptoS$
of two rectangles of $\rects$, then the induced orientations are the same.
This lets us define a relation $<$ on each rectangle $R$ of $\rects$,
where $p<q$ for $p,q\in R$ if $p$ is ``lower'' than $q$, that is, further from
the ``upper'' edge of $R$.

For each $X\subset\patch_\rects$, let
\begin{align*}
U_X &:= \Big\{ (p,q)\in\unordereds \mid
   \textrm{$\maptoS(p)\in X$ and $p<q$, or $\maptoS(q)\in X$ and $q<p$}
   \Big\},
\qquad\textrm{and} \\
D_X &:= \Big\{ (p,q)\in\unordereds \mid
   \textrm{$\maptoS(p)\in X$ and $q<p$, or $\maptoS(q)\in X$ and $p<q$}
   \Big\}.
\end{align*}

Define the set of horizontal segments:
\begin{align*}
\hat H &:=
   \Big\{ (p,q)\in\unordereds
             \mid \textrm{$p\not<q$, $q\not<p$, and $p\neq q$} \Big\}.
\end{align*}
We will also need the following subset of this set.
Let $H$ be the set of $(p,q)$ in $\hat H$
such that if $p$ and $q$ are in the same horizontal edge of a rectangle, then
both $\maptoS(p)$ and $\maptoS(q)$ are corner points.

Denote by $\delta_{(p,q)}$ the Dirac measure on $(p,q)\in\Omega$, that is,
the measure consisting of an atom of mass $1$ on $(p,q)$.

\begin{lemma}
\label{lem:balanced}
Let $\mu_n$ be a sequence in $\piecewisestraights$, and let $\lambda_n$ be a
sequence of positive real numbers such that $\lambda_n\mu_n$
converges to $\mu\in\probmeasures$ with $\mu[\hat H]=0$.
Then, for any patch $\rects$, we have $\mu[U_X]=\mu[D_X]$,
%for all Borel subsets $X$ of $S$ satisfying $\closure X\subset\patch_\rects$.
for all Borel subsets $X$ of $\patch_\rects$.
\end{lemma}
\begin{proof}
Fix $n\in\N$ and a Borel subset $X$ of $S$ satisfying
$\closure X\subset\patch_\rects$.
Since $\mu_n$ is in $\piecewisestraights$ it can be written
$\mu_n = \sum_{j=1}^{|\mu_n|} \delta_{(p_j,q_j)}$, with the ordered pairs
$\{(p_j,q_j)\}_j$ satisfying (i)--(iv). Define the sets
\begin{align*}
U^+ &:= \{ j \mid \textrm{$r(p_j)\in X$ and $p_j<q_j$} \}, \\
U^- &:= \{ j \mid \textrm{$r(q_j)\in X$ and $q_j<p_j$} \}, \\
D^+ &:= \{ j \mid \textrm{$r(p_j)\in X$ and $q_j<p_j$} \}, \\
D^- &:= \{ j \mid \textrm{$r(q_j)\in X$ and $p_j<q_j$} \},
   \qquad\textrm{and} \\
H^\pm &:= \{ j \mid \textrm{$(p_j,q_j)\in H$} \}.
\end{align*}
From (i) and (ii), we see that if $j$ is in $U^+$,
then $j-1$ is in either $D^-$ or $H^\pm$.
Similarly, if $j$ is in $U^-$, then $j+1$ is in either $D^+$ or $H^\pm$.
So,
\begin{align*}
\mu_n[U_X]
   &= \sharp U^+ + \sharp U^- \\
   &\le \sharp D^- + \sharp H^\pm + \sharp D^+ + \sharp H^\pm \\
   &= \mu_n[D_X] + 2\mu_n[H].
\end{align*}
Here, ``$\sharp $'' denotes the number of elements in a set.
A similar inequality with $U_X$ and $D_X$ interchanged can also be derived in
the same way. We conclude that
\begin{align}
\label{eqn:diff_UX_DX}
\Big|\mu_n[U_X] - \mu_n[D_X] \Big| \le 2 \mu_n[H].
%\qquad\textrm{for all $n\in\N$ and $X\in Z$.}
\end{align}
%for all $n\in\N$ and $X\in\borels(\patch)$.

Let $\partial X:=\closure X \backslash \interior X$ be the boundary of $X$.
By assumption, $\partial X\subset\patch_\rects$. We have
\begin{align*}
\boundary U_X &\subset \hat H \union P \union U_{\boundary X},
\qquad\textrm{and} \\
\boundary D_X &\subset \hat H \union P \union D_{\boundary X}.
\end{align*}

Let $Z$ be the set of Borel subsets $X$ of $S$ such that
$\closure X\subset\patch_\rects$ and
$\mu[U_{\boundary X}]=\mu[D_{\boundary X}]=0$.

By assumption, $\mu[\hat H]=0$, so using Lemma~\ref{lem:no_degenerate_atoms},
we get that $\mu[P]=0$. Therefore,
we may apply the Portmanteau theorem to get that $\mu_n[U_X]$ and
$\mu_n[D_X]$ converge, respectively, to $\mu[U_X]$ and $\mu[D_X]$, for all
$X\in Z$. Also, since $H$ is closed, $\limsup_n \mu_n[H]\le\mu[H]=0$.
We see therefore that $\mu[U_X]=\mu[D_X]$ for all $X\in Z$.

Both $\mu[U_\cdot]$ and $\mu[D_\cdot]$ are finite measures on
$\patch_\rects$. Since, for any subsets $X$ and $Y$ of $\patch_\rects$,
one has $\boundary(X\intersection Y) \subset\boundary X\union\boundary Y$,
we have that $Z$ is closed under finite intersections.

Take $G$ to be an open Borel subset of the space $\patch_\rects$.
Choose some metric $d$ on $S$ compatible with the topology, and define,
for each $\epsilon\in(0,1)$,
\begin{align}
G_\epsilon := \Big\{ x\in G \mid d(x,\boundary G) \ge \epsilon \Big\},
\end{align}
where $\boundary G$ denotes the boundary of $G$ in $\patch_\rects$.
Since every point $x$ of $\partial G_\epsilon$ satisfies
$d(x,\boundary G)=\epsilon$, the sets $\{\partial G_\epsilon\}_\epsilon$ are
pairwise disjoint. Therefore,
only countably many such sets satisfy $\mu[U_{\boundary G_\epsilon}]>0$
and only countably many satisfy $\mu[D_{\boundary G_\epsilon}]>0$.
So, $G_\epsilon$ is in $Z$ for $\epsilon$ in some dense subset of $(0,1)$.
Hence, $G$ can be written as a countable union of elements of $Z$.

We have shown that $Z$ is a $\pi$-system that generates the Borel
$\sigma$-algebra of $\patch_\rects$. So, since the measures $\mu[U_\cdot]$ and
$\mu[D_\cdot]$ agree on $Z$, they agree on every Borel subset of
$\patch_\rects$. This concludes the proof.
\end{proof}

Given a patch $\rects$, define $F_Y:= \{(p,q)\in\Omega
   \mid \text{$p<q$ and $[p,q)\intersection Y\neq\emptyset$}\}$,
for any $Y\subset\patch_\rects$.

Let $G$ be an unmeasured foliation. A \emph{generalised transverse measure}
$\mu$ on $G$ is a map associating a measure to each transverse arc that does
not pass through a singular point, with the following condition:
if $\alpha:[0,1]\to S$ and $\beta:[0,1]\to S$ are two such arcs that are
isotopic through transverse arcs whose endpoints remain in the same leaf,
then $\mu(\alpha)=\mu(\beta)$.
We do not require that the measure is regular with respect to the
Lebesgue measure.

\newcommand\homotopy{\mathcal{I}}

\begin{lemma}
\label{lem:generalised_transverse}
Assume that, for every patch $\rects$, a measure $\mu\in\probmeasures$
satisfies $\mu[U_X]= \mu[D_X]$ for all Borel $X\subset\patch_\rects$,
and that $\int h\,\dee\mu=0$.
Then, there exists a generalised transverse measure $\gentran$ on $G$ such that
\begin{align}
\mu[F_{r^{-1}(\alpha)}] = \int_\alpha \mathrm{d}\gentran,
\end{align}
for every short transverse arc $\alpha$.
\end{lemma}
\begin{proof}
For any transverse arc $\alpha$ avoiding singularities,
write $\alpha$ as a concatenation of short transverse arcs $\{\alpha_j\}_j$,
and define
\begin{align*}
\int_\alpha \mathrm{d}\gentran := \sum_j \mu[F_{r^{-1}(\alpha_j)}],
\end{align*}
where each $F_{r^{-1}(\alpha_j)}$ is relative to some patch containing
$\alpha_j$, which we fix.
That the same result is obtained when one takes a different decomposition
of $\alpha$ can be seen by considering a common refinement of the two
decompositions.

We must show that $\gentran$ is a generalised transverse measure. Let
$\alpha_0, \alpha_1:[t_0,t_1]\to S$ be two transverse arcs isotopic through
an isotopy $\homotopy:[t_0,t_1]\times[0,1]\to S$, along which each point stays
in the same leaf. We write $\alpha_s := \homotopy(\cdot,s)$,
for all $s\in[0,1]$.
Since $\homotopy$ is continuous, $\alpha_{s'}$ converges uniformly to
$\alpha_{s}$ as $s'$ tends to $s$; see~\cite[Lemma~3.1]{walsh_stretch}.

Let $s\in[0,1]$. Write $\alpha_s$ as a concatenation of short transverse arcs
$\{\alpha_s^j\}_j$, where the domains are pairwise disjoint intervals
$\{I_j\}_j$ satisfying $\union_j I_j = [t_0,t_1]$.

Fix $j$. For $s'$ close enough to $s$, the arcs $\alpha_s^j$ and $\alpha_{s'}$
restricted to $I_j$ are in the image under $r$ of the same rectangular
patch $\rects$.

Fix such an $s'$. Recall that we may define a notion of ``upwards'' on
$\patch_\rects$. Decompose $I_j$ into three sets $T^0$, $T^-$, and $T^+$,
depending on whether $\alpha_{s'}(t)$ is identical to, below, or above
$\alpha_s(t)$, respectively, for $t\in I_j$.

For notational convenience, we write $F_W:= F_{r^{-1}(\alpha_s(W))}$ and
$F'_W:= F_{r^{-1}(\alpha_{s'}(W))}$, for $W\subset I_j$.

Clearly $F_{T^0} = F'_{T^0}$.

Let $X$ denote the union over $T^+$ of the half-open leaf segments
$(\alpha_s(t),\alpha_{s'}(t)]$, and let $Y$ denote the union over $T^+$ of the
open leaf segment that starts on $\alpha_{s'}(t)$, ends on the boundary of
$\patch_\rects$,and does not contain $\alpha_s(t)$.
%FIGURE!!  ; see Figure~\ref{fig:}.

Since $\mu$ is supported on $V$,
we have $F'_{T^+} \backslash F_{T^+} = U_X \intersection D_Y$.
So
\begin{align*}
(U_X \intersection D_X) \union (F'_{T^+} \backslash F_{T^+})
   = U_X \intersection (D_X \union D_Y)
   = U_X.
\end{align*}
Similarly,
\begin{align*}
(U_X \intersection D_X) \union (F_{T^+} \backslash F'_{T^+})
   = D_X.
\end{align*}
But $U_X \intersection D_X$ is disjoint from both $F_{T^+}$ and $F'_{T^+}$,
and, by assumption, $\mu[U_X]=\mu[D_X]$.
We deduce that $\mu[F_{T^+}]=\mu[F'_{T^+}]$.

One may deduce in a similar fashion that $\mu[F_{T^-}]=\mu[F'_{T^-}]$.
So, we have proved that  $\mu[F_{I_j}]=\mu[F'_{I_j}]$.

Since this works for each $j$, we see that, for all $s'$ in a some
neighbourhood of $s$, the transverse lengths with respect to $\gentran$
of $\alpha_s$ and $\alpha_{s'}$ are equal. Using that $s$ was chosen
arbitrarily and that $[0,1]$ is connected,
we get that $\int_{\alpha_s} \mathrm{d}\tilde\mu$ is independent of~$s$.
\end{proof}

\newcommand\Fdash{\tilde F}

For the next two lemmas, we will need the following notation.
Given a patch $\rects$, define
%For any $Y\subset\unionrects$, define
%\begin{align*}
%F'_Y := \{(p,q)\in\unordereds \mid [p,q] \intersection Y \neq \emptyset\}.
%\end{align*}
$\Fdash_Y := \{(p,q)\in\unordereds \mid [p,q]\intersection Y \neq \emptyset\}$,
for all $Y\subset\patch_\rects$.
For two parameterised closed curves or arcs $\alpha$ and $\beta$ on a surface,
we define $\sharp (\alpha,\beta)$ to be the cardinal number of the set
$\{(s,t)\mid\alpha(s)=\beta(t)\}$.
By a straight arc on $S$, we mean one that is straight in the singular flat
metric associated to a given quadratic differential.

\begin{lemma}
\label{lem:estimate_intersection}
Suppose a rectangulation is given. Let $\mu\in \piecewisestraights$, and let
$\beta$ be a closed curve that can be written as a concatenation of a finite
number of straight short arcs $\{\beta_j\}_j$.
%none of which have both endpoints on the image of the boundary of a rectangle.
Then, $i(\curveof{\mu},\beta)\le \sum_j \mu[\Fdash_{r^{-1}(\beta_j)}]$.
\end{lemma}
\begin{proof}
If $x\in\surface$ is such that no element of $r^{-1}(x)$ lies on $[p,q]$
for some $(p,q)\in\unordereds$, then there is some neighbourhood of $x$ all
of whose points have the same property. One may use this to show that any
sufficiently small perturbation of the straight short arcs $\{\beta_j\}$
will not increase $\sum_j \mu[\Fdash_{r^{-1}(\beta_j)}]$.

Suppose that $\sharp (r[p,q],\beta_j)$ is infinite for some atom $(p,q)$
of $\mu$ and some $j$. Then, we may perturb an endpoint of $\beta_j$ so that
$\sharp (r[p,q],\beta_j)$ becomes either zero or one and
$\sum_j \mu[\Fdash_{r^{-1}(\beta_j)}]$ is not increased.
So we may assume, without loss of generality, that
$\sharp (r[p,q],\beta_j)$ is either zero or one for all atoms $(p,q)$ of $\mu$
and all $j$. Note that in this case
$\delta_{(p,q)}[\Fdash_{r^{-1}(\beta_j)}]=\sharp (r[p,q],\beta_j)$.
Write $\mu=\sum_k\delta_{(p_k,q_k)}$.
So,
\begin{align*}
i(\curveof{\mu},\beta)
   &\le \sharp (\curveof{\mu},\beta) \\
   &\le \sum_j \sum_k \sharp (r[p_k,q_k],\beta_j) \\
   &=    \sum_j \mu[\Fdash_{r^{-1}(\beta_j)}].
\qedhere
\end{align*}
\end{proof}

\newcommand\transv{\mathcal{P}}
%Let $\transv$ be the set of curves that are piecewise transverse and straight,
%whose endpoints do not lie on atoms of $\mu$.

\begin{lemma}
\label{lem:intersection_beta_mu}
Suppose that a rectangulation is given.
Let $\mu_n$ be a sequence in $\piecewisestraights$, and $\lambda_n$ be
a sequence of positive real numbers.
Assume that $\lambda_n\curveof{\mu_n}$ converges to $F\in\measuredfoliations$,
and that $\lambda_n\mu_n$ converges to $\mu\in\probmeasures$ satisfying
$\int h\,\dee\mu=0$.
Then, $i(F,\beta) \le \int_\beta\dee\tilde\mu$, for all closed curves $\beta$
avoiding singularities.
\end{lemma}
\begin{proof}
Since $\lambda_n[\curveof{\mu_n}]$ converges to $F$,
we have that $\lambda_n i(\curveof{\mu_n},\beta)$ converges to $i(F,\beta)$.
Perturb $\beta$ so that it is a concatenation of closed straight transverse
short arcs $\beta_j$, and so that $\mu[D_{\{x\}}]=0$ for all points
$x\in\union_j\beta_j$.
We may do this in such a way that $\int_\beta \dee\tilde\mu$ is not increased
by more than an arbitrarily small $\epsilon>0$.
By Lemma~\ref{lem:estimate_intersection},
$i(\curveof{\mu_n},\beta)\le \sum_j \mu_n[\Fdash_{r^{-1}(\beta_j)}]$,
for all $n$.

Each set $\Fdash_{r^{-1}(\beta_j)}$ is closed, and so
\begin{align*}
\limsup_n\lambda_n\mu_n[\Fdash_{r^{-1}(\beta_j)}]
   \le \mu[\Fdash_{r^{-1}(\beta_j)}],
\qquad\text{for each~$j$.}
\end{align*}
We also have $\mu[\Fdash_{r^{-1}(\beta_j)}] = \mu[F_{r^{-1}(\beta_j)}]$,
for all $j$.
Applying Lemmas~\ref{lem:balanced} and~\ref{lem:generalised_transverse},
we see that $\mu[F_{r^{-1}(\beta_j)}] = \int_{\beta_j} \mathrm{d}\tilde\mu$,
for all $j$.
Putting all of this together, and using the fact that $\epsilon$
is arbitrary, we deduce the result.
\end{proof}

\subsection{Generalised transverse measures with no atoms}

\newcommand\gtms{\mathcal{G}}

Suppose we are given an unmeasured foliation $G$.
Consider the set of generalised
transverse measures on $G$ that have no atoms on saddle connections.
We regard two of them as being equivalent if they agree on each minimal
component of $G$ and give the same height to each annular component.
Let $\gtms(G)$ be the space of equivalence classes under this equivalence
relation. We see that $\gtms(G)$ is a closed finite-dimensional cone.
There is one extremal ray for each annular component of $G$, and one for
each projective class of ergodic transverse measure on a minimal component.
Let a set $J$ index these extremal rays, and, for each $j\in J$, choose an
element $\nu_j\in\gtms(G)$ of the $j$th extremal ray.
Every $\nu\in\gtms(G)$ can be written $\nu=\sum_{j\in J}f_j \nu_j$,
for some collection of non-negative coefficients $\{f_j\}_{j\in J}$.

Any element of $\gtms(G)$ gives rise to an element of
$\measuredfoliations$.
We define the intersection number of a generalised transverse measure
$\tilde\mu\in\gtms(G)$ and a curve class $\beta$ in $\curves$ to be
$i(\tilde\mu,\beta):=\inf_\beta\int_\beta\dee\tilde\mu$,
where the infimum is taken over all curves in the curve class.
Clearly, $i(\cdot,\beta)$ is linear for fixed $\beta$.

The following lemma is~\cite[Lemma~6.3]{walsh_stretch}, restated in terms
of measured foliations rather than measured laminations.
\begin{lemma}
\label{lem:intersect_one}
Let $\{F_j\};j\in\{0,\dots,n\}$ be a set of projectively-distinct
indecomposable elements of
$\measuredfoliations$ such that $i(F_j,F_k)=0$ for all $j$ and $k$,
and let $\epsilon>0$. Then, there exists a curve class $[\beta]$ in $\curves$
such that $i(F_j,\beta)< i(F_0,\beta) \epsilon$, for all $j\neq 0$.
\end{lemma}

\begin{lemma}
\label{lem:smaller_coeffs}
Let $\tilde\mu\in\gtms(G)$ be written $\tilde\mu=\sum_j f_j\nu_j$.
Let $F\in\measuredfoliations$ be such that $i(F,\beta)\le i(\tilde\mu,\beta)$,
for all $\beta\in\curves$. Then $F$ has a representation of the form
$(G,\tilde\mu')$, where $\tilde\mu'=\sum_j f'_j\nu_j$ with non-negative
coefficients $\{f'_j\}$. Moreover, $f'_j\le f_j$, for all~$j$.
\end{lemma}
\begin{proof}
Since intersection number is continuous, we have $i(F,H)\le i(\tilde\mu,H)$, 
for all $H\in\measuredfoliations$. In particular, taking $H:=(G,\tilde\mu)$,
we get $i(F,\tilde\mu)=0$. If $F$ had an indecomposable component $F'$
that is not a multiple of $(G,\nu_j)$ for any $j$, then we could use
Lemma~\ref{lem:intersect_one} to find a curve $\beta\in\curves$ such that
$i(F',\beta)$ is much greater than $i(f_j\nu_j,\beta)$, for all $j$.
However, this is impossible by assumption. Hence, $F$ can be written
$F=(G,\tilde\mu')$, where $\tilde\mu'=\sum_j f'_j\nu_j$ with non-negative
coefficients $\{f'_j\}$.

Choose $\epsilon>0$, and let $\tilde\mu^c:=\sum_{k\neq j}f_k\nu_k$.
By Lemma~\ref{lem:intersect_one}, there is a curve class $\beta\in\curves$
such that $i(\tilde\mu^c,\beta)< i(f_j\nu_j,\beta)\epsilon$.
Therefore,
\begin{align*}
(1+\epsilon) f_j i(\nu_j,\beta)
   > i(\tilde\mu,\beta) 
   \ge i(\tilde\mu',\beta)
   \ge f'_j i(\nu_j,\beta).
\end{align*}
The result follows since $\epsilon$ is arbitrary.
\end{proof}

\subsection{Construction of a weighted rectangulation.}

In this subsection, we construct a particular weighted rectangulation.

We will need the following measure theoretic lemma.
\begin{lemma}
\label{lem:subintervals}
Let $\{\nu_j\}$; $j\in\{0,\dots,n\}$ be mutually-singular non-atomic finite
measures on an interval $I$ of the real line.
Then, for any $\delta>0$, there exists a
decomposition of $I$ into disjoint subsets $P_j$; $j\in\{0,\dots,n\}$ such that
each $P_j$ is composed of a finite number of intervals,
and $\nu_j[I]-\nu_j[P_j]<\delta$, for all $j$.
\end{lemma}
\begin{proof}
The measures $\{\nu_j\}$ are mutually singular, so there exists a decomposition
$I = X_0 \union \dots \union X_n$ of $I$ into disjoint Borel sets such that
$\nu_j[X_k]$ equals $\nu_j[I]$ when $j=k$, and is zero otherwise.

Fix $j\in\{0,\dots,n\}$, and let $\delta>0$ be given.
We may approximate $X_j$ from above by open sets of $I$:
\begin{align*}
0=
\sum_{k\neq j} \nu_k[X_j]
   = \inf\Big\{ \sum_{k\neq j} \nu_k[U_j] \mid
         \text{$X_j \subset U_j \subset I$, and $U_j$ is open}
     \Big\}.
\end{align*}
So, there is an open set $U_j$ of $I$ such that $X_j \subset U_j \subset I$
and $\sum_{k\neq j} \nu_k[U_j]<\delta$.
In particular, $\nu_k[U_j]<\delta$, for all $k\neq j$.

Since $U_j$ is open, it is the disjoint union of a countable collection
of open intervals. We can choose a finite number of these with union $V_j$
such that
\begin{align*}
\nu_j[V_j] + \delta > \nu_j[U_j] \ge \nu_j[X_j] = \nu_j[I].
\end{align*}
Do this for all $j$.

For all $j\in\{1,\dots,n\}$, let
\begin{align*}
P_j &:= V_j\backslash\bigunion_{k\neq j} V_k, \\
\text{and let}\qquad
P_0 &:= I \backslash\bigunion_{k\neq 0}P_k
    \supset V_0 \backslash\bigunion_{k\neq 0} V_k.
\end{align*}
Each $P_j$ is a finite disjoint union of intervals (not necessarily open).
The $\{P_j\}$ are pairwise disjoint.
Clearly,
\begin{align*}
\nu_j[P_j]\ge\nu_j[V_j] - \sum_{k\neq j} \nu_j[V_k],
\qquad\text{for all $j$}.
\end{align*}
However,
\begin{align*}
\sum_{k\neq j} \nu_j[V_k] \le \sum_{k\neq j} \nu_j[U_k]  \le n\delta,
\qquad\text{for all $j$}.
\end{align*}
The conclusion follows.
\end{proof}

\begin{construction}
\label{con:rectangulation}

Let $G$ be the unmeasured foliation obtained from $V(q)$ by forgetting the
transverse measure. As before, let $J$ be a set indexing the extremal rays
of $\gtms(G)$.
Suppose we are given, positive real numbers $\delta$, $\epsilon$,
and $\{\theta_j\};j\in J$.

Let $J_S\subset J$ be the subset of indices associated to annuli.
For each $j\in J_S$, let $A_j$ be the annulus.
By cutting $A_j$ along a horizontal leaf, we obtain a rectangle, to which
we give weight $\theta_j$.

Let $D$ be a minimal domain of $G$, and take a horizontal arc $I$ in the
interior of $D$. By considering the point of first return of leaves starting
on $I$, we obtain a (non-oriented) interval exchange map,
and hence a rectangular decomposition $\{R_l\}$ of $D$.
There is a one-to-one correspondence between the ergodic measures of the
interval exchange map and the indecomposable projective measured foliations
that are supported on $D$. Let $J_D\subset J$ index this set, and
let $\{\nu_j\}$; $j\in J_D$ be the ergodic measures of the
interval exchange map.
Consider one of the rectangles $R_l$. Take one of its horizontal edges $I_l$
and apply Lemma~\ref{lem:subintervals} to the measures $\{\nu_j\}_j$ restricted
to $I_l$. We get a decomposition of $I_l$ into $\sharp J_D$ disjoint sets
$\{P_{lj}\}_j$,
each the disjoint union of a finite number of intervals, such that
$\nu_j[I_l] - \nu_j[P_{lj}] < \delta$, for all $j\in J_D$.
By sweeping the edge $I_l$ along the rectangle $R_l$, this gives us a
decomposition of $R_l$ into sub-rectangles. Associate the weight $\theta_j$
to each rectangle swept out by an interval contained in $P_{lj}$.
We repeat this construction for all rectangles in $\{R_l\}$ and for all
minimal domains of $G$.
Since the annuli and the minimal components of $G$ make up the whole surface,
the construction so far gives us a weighted rectangulation $\Gamma$.
Let $\rho_\theta$ be the conformal metric on $S$ induced by this rectangulation.

However, we wish to give special treatment to the saddle connections.
To do this, we will define another rectangulation.
Let $C_\epsilon$ be the closure of the set of points of $S$ that are
connected to a non-singular point of the critical graph by a horizontal arc
of length less than $\epsilon^2$ in the singular flat metric coming from $q$.
%see Figure~\ref{fig:thicken_critical}
We assume that $\epsilon$ is small enough that this is a union of rectangles,
one for each saddle connection. Give weight $1/\epsilon$ to each of these
rectangles. The remaining part of the surface $S\backslash C_\epsilon$ may be
decomposed into rectangles, to which we give weight zero.
The conformal metric on $S$ induced by this rectangulation we denote by
$\rho_\epsilon$.

We combine the two rectangulations we have constructed as follows.
We say that one rectangulation $r_1:\unionrects_1\to\completionS$ is
\emph{finer} than another $r_2:\unionrects_2\to\completionS$ if for every
rectangle $R_1$ in $\unionrects_1$ there exists a rectangle $R_2$ in
$\unionrects_2$ such that $r_1(R_1) \subset r_2(R_2)$.
Given two rectangulations, one may find a third rectangulation that is
finer than both. On this rectangulation we may choose a weighting in such a way
that the conformal metric induced on $S$ is the sum of the conformal metrics
induced by the original rectangulations.
So we obtain a conformal metric $\rho=\rho_\theta+\rho_\epsilon$.
\end{construction}

Since $V(q)\in\gtms(G)$, we may write it as $V(q)=\sum_{j\in J}g_j \nu_j$,
where the $\{g_j\}$ are non-negative coefficients. Note that some of the
$\{g_j\}$ may be zero. We let $a_j:=i(g_j\nu_j,H(q))$, for all $j\in J$,
and consider it to be the area of the $j$th component of $V(q)$. Observe that
$\sum_{j\in J}a_j = i(V(q),H(q))$ is the area of the singular flat metric
associated to the quadratic differential $q$.

We use the notation $O(\epsilon)$ to stand for any function that is less than
than some linear function of $\epsilon$, for $\epsilon$ small enough.

\begin{lemma}
\label{lem:area_rho}
In construction~\ref{con:rectangulation}, fix the quadratic differential $q$,
the parameters $\theta_j$, and the choice of horizontal arc in each minimal
domain. Then, the area of the conformal metric $\rho$ obtained from
the constructed weighted rectangulation satisfies
$A(\rho) \le \sum_j\theta_j^2 a_j + O(\epsilon) + O(\delta)$.
\end{lemma}
\begin{proof}
We have $A(\rho) = A(\rho_\theta) + A(\rho_\epsilon)$.

First observe that the area of $\rho_\epsilon$ is $\epsilon^2 L/\epsilon$,
where $L$ is the total length of the critical graph with respect to $H(q)$.
This is $O(\epsilon)$.

Let $j\in J_S$. The corresponding annular component $A_j$ of $G$
contributes $\theta_j^2 a_j$ to the area of $\rho_\theta$.

Now consider a minimal domain $D$ of $G$, and let $J_D\in J$ be the set of
indices of the ergodic measured foliations supported on it. In the construction,
$D$ was decomposed into a finite number of rectangles $\{R_{l}\}$,
each having a horizontal edge $I_l$ that is further subdivided into
$\sharp J_D$ disjoint sets $\{P_{lj}\}$, each composed of a finite number
of intervals.

Observe that $a_j = \sum_l h_l \nu_j[I_l]$, for all $j\in J_D$,
where $h_l$ is the height of $R_l$ with respect to $H(q)$.

We have $\nu_k[P_{lj}] \le \nu_k[I_{l}] - \nu_k[P_{lk}] < \delta$,
for all distinct $j$ and $k$ in $J_D$. Therefore,
\begin{align*}
\nu[P_{lj}]
   &= \nu_j[P_{lj}] + \sum_{\text{$k\in J_D$, $k\neq j$}} \nu_k[P_{lj}] \\
   &\le \nu_j[I_l] + (\sharp J_D -1) \delta,
\qquad\text{for all $j\in J_D$}.
\end{align*}
So, the contribution of $D$ to the area $\rho_\theta$ is
\begin{align*}
\sum_l h_l \sum_{j\in J_D} \theta_j^2 \nu[P_{lj}]
   &\le \sum_{j\in J_D} \theta_j^2 a_j + O(\delta).
\qedhere
\end{align*}
\end{proof}

Define the $\rho$-length of a generalised transverse measure $\tilde\mu$ to be
\begin{align*}
\text{$\rho$-length}(\tilde\mu):=\int_S\rho(\dee\tilde\mu\times\dee H(q)).
\end{align*}

\begin{lemma}
\label{lem:rho_mu}
In construction~\ref{con:rectangulation}, fix the quadratic differential $q$,
the parameters $\theta_j$, and the choice of horizontal arc in each minimal
domain.
Let $\tilde\mu\in\gtms(G)$ be a generalised transverse measure on $G$ with no
atoms. Write $\tilde\mu=\sum_{j\in J}f_j \nu_j$.
Then the $\rho$-length of $\tilde\mu$ satisfies
$\text{$\rho$-length}(\tilde\mu)
   \ge \sum_j \theta_j f_j i(\nu_j,H(q)) - O(\epsilon) - O(\delta)$.
\end{lemma}
%\begin{proof}
%We have that
%$\length_\rho(\mu) \ge \length_{\rho_\theta}(\mu) + \length_{\rho_M}(\mu)$.
%
%Let $j\in J_S$, and let $I_j$ be a horizontal arc crossing the annulus
%associated to $j$. By Lemma~\ref{lem:annular_interval},
%$\mu[I_j] \ge \nu_j[I_j]f_j$. So, the contribution of the annulus to the
%$\rho\,$-length of $\mu$ is greater than or equal to
%$\theta_j c_j f_j \nu_j[I_j] = \theta_j f_j a_j$.
%(Recall that $c_j$ is the cicumference of the annulus.)
%
%Let $D$ be a minimal domain of $(G,\nu)$, and let $J_D$ be the indices
%of the ergodic measured foliations comprising it. Recall that
%$D$ is decomposed into a finite number of rectangles $\{R_{l}\}$,
%each having a horizontal edge $I_l$ that is further subdivided into $n+1$
%disjoint sets $\{P_{lj}\}$, each formed of a finite number of intervals.
%
%By Lemma~\ref{lem:min_interval},
%\begin{align*}
%\mu[P_{lj}] &\ge \sum_{k\in J_D} \nu_k[P_{lj}] f_k \\
%   &\ge \nu_j[P_{lj}]  f_j,
%\end{align*}
%for all $l$ and all $j\in J_D$.
%By construction, $\nu_j[P_{lj}] > \nu_j[I_{l}] - \epsilon$, for all $j\in J_D$.
%Therefore, the contribution of $D$ to $\length_{\rho_M}(\mu)$ is greater than
%\begin{align*}
%\sum_l h_l \sum_{j\in J_D} \theta_j \mu[P_{lj}]
%   &\ge \sum_{k\in J_D} \theta_j f_j\Big(\sum_l h_l \nu_j[I_{l}]\Big)
%           - O(\epsilon) \\
%   &\ge \sum_{j\in J_D} \theta_j f_j a_j - O(\epsilon).
%\qedhere
%\end{align*} 
%\end{proof}
\begin{proof}
The proof is similar to the proof of Lemma~\ref{lem:area_rho}.
\end{proof}

\begin{definition}
\label{def:dualextfunc}
Let
\begin{align*}
\dualextfunc_q([F]) :=
   \begin{cases}
   \sum_{j}f^2_j a_j,
            & \text{if $F=\sum f_j G_j$ with $V(q)=\sum G_j$}, \\
   +\infty, & \text{otherwise},
   \end{cases}
\end{align*}
$a_j:=i(G_j,H(q))$ is the area of the indecomposable
component $j$ of $V(q)$ relative to~$q$.
\end{definition}

Now we are ready to prove the key lemma of this section.

\begin{lemma}
\label{lem:bottom}
Let $R(q;\cdot)$ be ray in Teichm\"uller space with initial quadratic
differential $q$.
Let $F_n$ be a sequence in $\measuredfoliations$ converging to an element
$F$ of $\measuredfoliations$, and let $t_n$ be a sequence of times diverging
to infinity.
Then,
\begin{align*}
\liminf_{n\to\infty}
   {e^{2t_n}\extlength_{R(q;t_n)}[F_{n}]}
      \ge {\dualextfunc_q[F]}.
\end{align*}
\end{lemma}
\begin{proof}
Since $\curves$ is dense in $\projmeasuredfoliations$ and
$\extlength_{R(q;t_n)}[\cdot]$ is continuous for fixed $t$, there exists
a sequence $([\alpha_n])_n$ of curve classes, and a sequence
of positive real numbers $\lambda_n$ such that $\lambda_n[\alpha_n]$ converges
to $F$, and
\begin{align*}
\Big|e^{2t_n} \extlength_{R(q;t_n)}[\lambda_n \alpha_n]
   - e^{2t_n} \extlength_{R(q;t_n)}[F_n] \Big|
\longrightarrow 0,
\qquad\text{as $n\to\infty$}.
\end{align*}
%converges to zero, as $n$ tends to infinity.
So, to establish the lemma, it suffices to show that
\begin{align*}
L :=
   \liminf_{n\to\infty}
      e^{2t_n} \extlength_{R(q;t_n)}[\lambda_n \alpha_n]
   \ge \dualextfunc_\xi[F].
\end{align*}
By taking a subsequence if necessary, we may assume that
$e^{2t_n} \extlength_{R(q;t_n)}[\lambda_n \alpha_n]$ converges
to $L$.

Let $G$ be the unmeasured foliation obtained from $V(q)$ by forgetting
the measure, and let $J$ be a set indexing the extremal rays of
$\gtms(G)$. For each $j\in J$, choose an representative $\nu_j\in\gtms(G)$
of the $j$th extremal ray. Choose positive real numbers
$\delta$ and $\epsilon$, and $\{\theta_j\};j\in J$.
Using these parameters,
define the weighted rectangulation $(\{R_l\}_l,r,\rho)$ according to
construction~\ref{con:rectangulation}.
This gives us a conformal metric $\rho$ on the Riemann surface $R(q;0)$.

Recall that, for each $t\in\Rplus$, one goes from $R(q;0)$ to $R(q;t)$
by stretching the vertical foliation and shrinking the horizontal
foliation by a factor $e^t$. Let $\rho_t$ be the conformal metric on $R(q;t)$
obtained from $\rho$ by stretching the surface in this way.
The area of $\rho_t$ is identical to that of $\rho$, for all $t\in\Rplus$,
because the stretching in the vertical direction is compensated by the
shrinking in the horizontal direction.

From the analytic
definition of extremal length,
\begin{align*}
\extlength_{R(q;t_n)}[\alpha_n]
   \ge \inf_{\alpha\in[\alpha_n]} \frac{\length_{\rho_{t_n}}(\alpha)^2}
                 {A(\rho_{t_n})},
\qquad\text{for all $n\in\N$.}
\end{align*} 
For each $n\in\N$, choose a representative $\alpha_n$ of $[\alpha_n]$
in such a way that
\begin{align*}
\lambda_n^2{e^{2t_n}}
   \Big|\length_{\rho_{t_n}}(\alpha_n)^2 - \inf_{\alpha\in[\alpha_n]}
                    \length_{\rho_{t_n}}(\alpha)^2 \Big|
\longrightarrow 0,
 \qquad\text{as $n\to\infty$.}
\end{align*}
Choose a sequence $\epsilon'_n$ of positive real numbers converging to zero.
We apply Lemma~\ref{lem:straighten} to get a sequence $\mu_n$ in $A$
such that $\curveof{\mu_n}$ is homotopic to $\alpha_n$, and
\begin{align*}
\int \rho v \,\mathrm{d}\mu_n
   &\le \int_{\alpha_n} \rho \,\mathrm{d}H(q) + \epsilon'_n,
\qquad\text{and} \\
\int h \,\mathrm{d}\mu_n
   &\le \int_{\alpha_n} \mathrm{d}V(q) + \epsilon'_n,
\end{align*}
By taking a subsequence if necessary, we may assume that the sequence of
probability measures $\mu_n / |\mu_n|$ converges weakly to a probability
measure $\mu$.

If $L$ is infinite, then there is nothing to prove. So, assume that $L$
is finite.

We claim that $\lambda_n|\mu_n|$ can not converge to infinity.
If it has some subsequence that converges to zero, then obviously the claim is
true. So, consider the case where it is bounded away from zero.
We have
\begin{align*}
\lambda_n e^{t_n} \length_{\rho_{t_n}}(\alpha_n)
   &\ge \lambda_n e^{2t_n} \int_{\alpha_n}\rho\,\dee V(q) \\
%   &\ge \lambda_n e^{t_n}
%           (\min\rho) \Big(\int h\,\dee\mu_n - \epsilon'\Big) \\
   &\ge \lambda_n |\mu_n| e^{2t_n}
           (\min\rho) \Big(\int h\,\dee\mu_n - \epsilon'_n\Big) / |\mu_n|,
\end{align*}
for $n$ large enough.
Using that $L$ is finite and that $e^{2t_n}$ grows without bound,
we see that
\begin{align*}
\int h\,\dee\mu = \lim_{n\to\infty} \int h\,\dee\frac{\mu_n}{|\mu_n|} =0.
\end{align*}
In particular, $\mu[\hat H]=0$, and so by Lemma~\ref{lem:no_degenerate_atoms}
we get $\mu[P]=0$.
Since $\mu$ is a probability measure, we deduce that $\int v\,\dee\mu >0$.
However,
\begin{align*}
\lambda_n e^{t_n} \length_{\rho_{t_n}}(\alpha_n)
   &\ge \lambda_n \int_{\alpha_n}\rho\,\dee H(q) \\
   &\ge \lambda_n |\mu_n|
           (\min\rho) \Big(\int v\,\dee\mu_n - \epsilon'_n\Big)/|\mu_n|.
\end{align*}
So, again using that $L$ is finite, we see that $\lambda_n |\mu_n|$ can not
converge to infinity.

Therefore, by restricting to a subsequence if necessary, we may assume that
$\lambda_n \mu_n$ converges weakly to a finite measure $\mu'$, which of
course will be a multiple of~$\mu$.
A similar argument to that just given shows that $\int h\,\dee\mu'=0$
and that $L\ge (\int \rho v\,\dee\mu')^2/A(\rho)$.

Applying Lemmas~\ref{lem:balanced}, \ref{lem:generalised_transverse},
and~\ref{lem:intersection_beta_mu}, we obtain a generalised transverse measure
$\gentran^\epsilon$ on $G$ such that
\begin{align}
\label{eqn:lowerboundcurve}
i(F,\beta)\le\int_\beta \dee\gentran^\epsilon,
\end{align}
for every closed curve $\beta$ avoiding singularities.
Moreover,
\begin{align*}
L\ge\frac{(\text{$\rho$-length}(\tilde\mu^\epsilon))^2}{A(\rho)}.
\end{align*}

We have made the dependence on $\epsilon$ explicit because we will now
let $\epsilon$ approach zero. Since $\int\rho v\,\dee\mu^\epsilon$
is bounded above uniformly in $\epsilon$, so also is
$\int_\alpha\dee\gentran^\epsilon$ for all transverse arcs $\alpha$ avoiding
singularities. We conclude that there is a sequence $\epsilon_n$ converging
to zero and a generalised transverse measure $\gentran$ such that
$\gentran^{\epsilon_n}$ converges to $\gentran$ as $n$ tends to infinity,
in the sense of weak convergence.

The contribution of a rectangle $R$ containing a saddle connection to the
$\rho^\epsilon$-length of $\gentran^\epsilon$ is
$\epsilon^{-1} h_R\int_{\beta^\epsilon}\dee\gentran^\epsilon$,
where $\beta^\epsilon$ is an open transverse arc crossing the rectangle, and
$h_R$ is the height of $R$ with respect to $H(q)$.
Thus, $\int_{\beta^\epsilon}\dee\gentran^\epsilon$, the mass of $\gentran^\epsilon$
crossing this rectangle, converges to zero as $\epsilon$ tends to zero.
From this and the properties of $\gentran^\epsilon$, we deduce that
$\gentran$ has no atoms. This implies that the $\rho^{\epsilon_n}$-length
of $\gentran^\epsilon$ converges to the $\rho^0$-length of $\gentran$.

From~(\ref{eqn:lowerboundcurve}), we get that
$i(F,\beta)\le\int_\beta\dee\gentran$, for every closed curve $\beta$.
Hence, $i(F,\beta)\le i(\gentran,\beta)$, for all $\beta\in\curves$.
We apply Lemma~\ref{lem:smaller_coeffs} to get that we may write
$F=(G,\tilde\mu')$, where $\tilde\mu'=\sum_j f'_j \nu_j$ with non-negative
coefficients $\{f'_j\}$ satisfying $f'_j\le f_j$, for all $j\in J$.
By Lemma~\ref{lem:rho_mu},
\begin{align*}
\text{$\rho^0$-length}(\gentran)
   \ge \sum_j\theta_j f_j i(\nu_j,H(q)) - O(\delta).
\end{align*}
By Lemma~\ref{lem:area_rho},
\begin{align*}
A(\rho) \le \sum_j\theta_j^2 g_j i(\nu_j,H(q)) + O(\delta).
\end{align*}
Therefore,
\begin{align*}
L \ge \frac{\big(\sum_j\theta_j f_j i(\nu_j,H(q)) \big)^2}
           {\sum_j\theta_j^2 g_j i(\nu_j,H(q))},
\end{align*}
where we have used the fact that $\delta$ is arbitrary.
Using the fact that the $\{\theta_j\}$ are also arbitrary, and applying
Lemma~\ref{lem:optimise}, we get that
$L\ge \sum_j f_j^2 a_j/g_j^2 = \dualextfunc_q[F]$.
\end{proof}

\begin{lemma}
\label{lem:flip}
For all measured foliations $F$ and quadratic differentials $q$,
\begin{align}
\label{eqn:flip_eqn}
\extfunc^2_q[F] = \sup_{F'\in\measuredfoliations\backslash\{0\}}
                    \frac{i(F,F')^2}{\dualextfunc_q[F']}.
\end{align}
\end{lemma}
\begin{proof}
Let $V(q)=\sum_j G_j$ be the decomposition of $V(q)$, the vertical foliation
of $q$, into indecomposable components.
Looking at the definition of $\dualextfunc_q[F']$,
we see that the right-hand-side of~(\ref{eqn:flip_eqn}) equals
\begin{align*}
\sup_{f\in\Rplus^n\backslash\{0\}}
      \frac{(\sum_j i(F,G_j)f_j)^2}{\sum_j f_j^2 a_j},
\end{align*}
where $a_j:= i(G_j,H(q))$ for each $j$.
By Lemma~\ref{lem:optimise}, this supremum is equal to
$\sum_j {i(F,G_j)^2}/{a_j}$,
as required.
\end{proof}

The following is the main result of this section.

\begin{lemma}
\label{lem:upper_boundary_along_ray}
Let $R(q;\cdot)$ be the Teichm\"uller ray with initial quadratic
differential~$q$. Then,
\begin{align*}
%\limsup_{t\to\infty} \extlength_{\gamma(t)}[F] / K_t \le \extfunc_\xi([F]).
%\limsup_{t\to\infty} \frac{\extlength_{\gamma(t)}[F]}{e^t}
\limsup_{t\to\infty} e^{-2t}\extlength_{R(q;t)}[F]
   \le \extfunc^2_q([F]),
\qquad\textrm{for all $F\in\measuredfoliations$}.
\end{align*}
\end{lemma}
\begin{proof}
Take a sequence of times $t_n$ such that
\begin{align*}
\limsup_{t\to\infty} e^{-2t} \extlength_{R(q;t)}[F]
   = \lim_{n\to\infty} e^{-2t_n} \extlength_{R(q;t_n)}[F]
\end{align*}
By Lemma~\ref{lem:inversion_formula}, for each $t\in\Rplus$, there exists
$[F'_t]\in\projmeasuredfoliations$ such that
\begin{align*}
\extlength_{R(q;t)}[F] \extlength_{R(q;t)}[F'_t] = {i(F,F'_t)^2}.
\end{align*}
Let $[F']\in\projmeasuredfoliations$ be a limit point of $[F'_{t_n}]$,
and choose representatives such that $F'_{t_n}$ converges to $F'$.
Using Lemma~\ref{lem:bottom} and the continuity of $i(F,\cdot)^2$, we get
\begin{align*}
\limsup_{n\to\infty}
   \frac{i(F,F'_{t_n})^2}{e^{2t_n}\extlength_{R(q;t_n)}[F'_{t_n}]}
      \le \frac{i(F,F')^2}{\dualextfunc_q[F']}.
\end{align*}
So,
\begin{align*}
\limsup_{t\to\infty} e^{-2t} \extlength_{R(q;t)}[F]
   &= \lim_{n\to\infty}
         \frac{i(F,F'_{t_n})^2}{e^{2t_n}
             \extlength_{R(q;t_n)}[F'_{t_n}]} \\
   &\le \frac{i(F,F')^2}{\dualextfunc_q[F']}.
\end{align*}
The result now follows on applying Lemma~\ref{lem:flip}.
\end{proof}

We may now prove Theorem~\ref{thm:extreme_len_asymptotics}

\begin{proof}[Proof of Theorem~\ref{thm:extreme_len_asymptotics}]
The result follows on combining Lemmas~\ref{lem:finite_time_bound}
and~\ref{lem:upper_boundary_along_ray}.
\end{proof}

\section{The horofunction boundary}
\label{sec:horofunction_boundary}

We recall the definition of the horofunction boundary of a metric space,
which first appeared in~\cite{gromov:hyperbolicmanifolds}.
See also~\cite{BGS} for more information.

Let $(X,d)$ a metric space. Choose a basepoint $b\in X$, and to each point
%, recall that if $(X,d)$ is an unbounded locally-compact
$z\in X$ associate the function $\phi_{z}\colon X\to\mathbb{R}$, with
\begin{align*}
\phi_{z}(x)=d(x,z)-d(b,z)
\qquad\mbox{for $x\in X$}.
\end{align*}
Assume that $(X,d)$ is \emph{proper}, meaning that closed balls are compact,
and geodesic, meaning that every pair of points is connected by a geodesic
segment. Under these assumptions, the map $\Phi\colon X\to C(X)$ given by
$\Phi(z)=\phi_{z}$ embeds $X$ into the space of continuous functions on $X$,
which is endowed with the topology of uniform convergence on bounded subsets
of $X$. We identify $X$ with its image under this embedding.
The \emph{horofunction boundary} of $X$ is defined to be
\[
X(\infty)=(\closure{\Phi(X)})\setminus\Phi(X),
\]
and its members are called \emph{horofunctions}.
Under our assumptions on $(X,d)$, the space $X\cup X(\infty)$ is a
compactification of $X$.

It is easy to verify that choosing a different base-point $b$ just has the
effect of altering each horofunction by an additive constant, and that the
horofunction boundaries coming from different basepoints are homeomorphic.

A path $\gamma\colon\Rplus\to X$ is called an \emph{almost-geodesic} if,
for each $\epsilon>0$,\index{almost-geodesic}
\begin{equation*}
|\dist(\gamma(0),\gamma(s))+\dist(\gamma(s),\gamma(t))-t|<\epsilon,
\text{\quad for $s$ and $t$ large enough, with $s\le t$}.
\end{equation*}
Rieffel~\cite{rieffel_group} proved that
every almost-geodesic converges to a limit in $X(\infty)$.
A horofunction is called a \emph{Busemann point}\index{Busemann point}
if there exists an almost-geodesic converging to it.
We denote by $X_B(\infty)$ the set of all Busemann points in $X(\infty)$.

For any two horofunctions $\xi$ and $\eta$, we define the \emph{detour cost} by
\begin{align*}
H(\xi,\eta) &=\sup_{W\ni\xi} \inf_{x\in W} \Big( d(b,x)+\eta(x) \Big),
\end{align*}
where the supremum is taken over all neighbourhoods $W$ of $\xi$ in the
compactification $X\cup X(\infty)$.
This concept originated in~\cite{AGW-m}.
An equivalent definition is
\begin{align}\label{eq:3.2}
H(\xi,\eta) &= \inf_{\gamma} \liminf_{t\to\infty}
                    d(b,\gamma(t))+\eta(\gamma(t)), 
\end{align}
where the infimum is taken over all paths $\gamma:\Rplus\to X$ converging to
$\xi$.

One can show that a horofunction $\xi$ is a Busemann point if and only if
$H(\xi,\xi)=0$.
The following result is useful for calculating the detour cost;
see~\cite[Lemma 3.3]{walsh_minimum} and~\cite[Lemma 5.2]{walsh_stretch}.

\begin{proposition}
\label{prop:lim_along_geos}
Let $x\in X$, and let $\gamma$ be an almost-geodesic converging to a
Busemann point $\xi$. Then, for any horofunction $\eta$,
\begin{equation*}
\lim_{t\to\infty} d(x,\gamma(t)) + \eta(\gamma(t)) = \xi(x) + H(\xi,\eta).
\end{equation*}
\end{proposition}

By symmetrising the detour cost, the set of Busemann points can be equipped
with a metric. For $\xi$ and $\eta$ in $X_B(\infty)$, we define
\begin{equation}\label{eq:3.4}
\delta(\xi,\eta) = H(\xi,\eta)+H(\eta,\xi) 
\end{equation}
and call $\delta$ the \emph{detour metric}.
This construction appears in~\cite[Remark~5.2]{AGW-m}.
The function $\delta\colon X_B(\infty)\times X_B(\infty)\to [0,\infty]$
is a metric, which might take the value $+\infty$.
Note that we can partition $X_B(\infty)$ into disjoint subsets such that
$\delta(\xi,\eta)$ is finite for each pair of horofunctions $\xi$ and $\eta$
lying in the same subset. We call these subsets the \emph{parts} of the
horofunction boundary of $(X,d)$, and $\delta$ is a finite-valued metric
on each one.

The detour metric $\delta$ is independent of the base-point.
Isometries of $(X,d)$ extend to homeomorphisms on the horofunction
compactification, and preserve the detour metric.

%This has appeared in~\cite[Remark~5.2]{AGW-m}.

\section{The horofunction boundary is the Gardiner--Masur boundary}
\label{sec:horoboundary_is_GM}

\newcommand\compx{\bar X}

We show in this section that the horofunction compactification of Teichm\"uller
space with the Teichm\"uller metric is just the Gardiner--Masur
compactification. This result has also appeared in the work of Liu and
Su~\cite{liu_su_compactification}. Our proof uses the bound from
Section~\ref{sec:lower_bound} but does not use any of the material from
Section~\ref{sec:upper_bound}.

A \emph{compactification} of a topological space $X$ is a pair
$(f,\compx)$, where $\compx$ is a compact topological space and
$f:X\to\compx$ is a homeomorphism onto its image, with $f(X)$ open and
dense in $\compx$.
Let $Y$ be a Hausdorff space. If $g$ is a continuous function from $X$
to $Y$, then we say that a function $\overline g$ from $\compx$ to $Y$
is a \emph{continuous extension} of $g$ to $\compx$ if
$\overline g\after f = g$.
A compactification $(f_1,X_1)$ of $X$ is said to be \emph{finer} than another
one $(f_2,X_2)$ if there exists a continuous extension of $f_2$ to $X_1$.
The two compactifications are said to be \emph{isomorphic} if each is finer
than the other.

\begin{lemma}
\label{lem:samecompactification}
Let $(f_1,X_1)$ and $(f_2,X_2)$ be two compactifications of $X$ such that
$f_2$ extends continuously to an injective map $g:X_1\to X_2$.
Then, the two compactifications are isomorphic.
\end{lemma}
\begin{proof}
Clearly, $X_1$ is finer than $X_2$.

We have
\begin{align}
\label{eqn:f2ing}
f_2(X)=g\after f_1(X)\subset g(X_1).
\end{align}
The denseness of $f_2(X)$ in $X_2$ gives that $\closure f_2(X) = X_2$.
Also, since $X_1$ is compact and $g$ is continuous, $g(X_1)$ is compact,
and hence closed in $X_2$.
Therefore, taking the closure of (\ref{eqn:f2ing}), we get $X_2\subset g(X_1)$.
So, $g$ is surjective. As a continuous bijection from a compact space to a
Hausdorff one, $g$ is a homeomorphism.
Its inverse satisfies $g^{-1}\after f_2=f_1$, and so is a continuous extension
of $f_1$ from $X_2$ to $X_1$.
\end{proof}

We will show that the Gardiner--Masur compactification and the horofunction
compactification are isomorphic by showing that each is isomorphic to a
third compactification.

%\begin{lemma}
%\label{}
%\end{lemma}
%\begin{proof}
%\end{proof}

\newcommand\compE{\bar E}

For each $x$ in $\teichmullerspace$, define
\begin{align*}
K_x := \sup_{F\in\unitfoliations}\frac{\extlength_x(F)}{\extlength_b(F)}.
\end{align*}
and $\extfunc_x:\measuredfoliations \to \Rplus:$
\begin{align*}
\extfunc_x(F) := \Big( \frac{\extlength_x(F)}{K_x} \Big)^{1/2}.
\end{align*}
Define $E:=\{ \extfunc_x \mid x\in\teichmullerspace\}$.
Let $\compE:=\closure E$ be its closure in the space of continuous
functions on $\measuredfoliations$ with the topology of uniform convergence
on compact sets.

Consider the compactification $(\extfunc,\compE)$,
where $\extfunc:x\mapsto \extfunc_x$.

\begin{lemma}
\label{lem:gardmasur}
The Gardiner--Masur compactification is isomorphic to the compactification
$(\extfunc,\compE)$.
\end{lemma}
\begin{proof}
Define the map
$\Psi:\compE \to \proj(\R^\curves), f\mapsto\Psi f:=[f|_\curves]$.
Here, $f|_\curves$ is the restriction of $f$ to the set $\curves$, and
$[\cdot]$ denotes projective equivalence class.
The map $\Psi$ is clearly continuous when we take on $\compE$
the topology of uniform convergence on compact sets, and on $\proj(\R^\curves)$
the quotient topology of the product topology.

For any $x\in\teichmullerspace$, we have
$\Psi\after\extfunc(x) = [\extfunc_x|_\curves] = [\extlength_x(\cdot)]$,
which is the vector in the Gardiner--Masur compactification associated
to the point $x$.

We conclude that $\Psi$ is a continuous extension of the map
$x\mapsto[\extlength_x(\cdot)]$.

Suppose that $\Psi f = \Psi g$ for some $f$ and $g$ in $\compE$.
This means that $f|_\curves = \lambda g|_\curves$ for some $\lambda>0$.
By continuity, $f = \lambda g$ on all of $\measuredfoliations$.
Taking the supremum over $\measuredfoliations$, we see that $\lambda=1$,
and so $f=g$. This proves that $\Psi$ is injective.

We now apply Lemma~\ref{lem:samecompactification}.
\end{proof}

For each $f\in\compE$, let $\Psi f$ be the function from
$\teichmullerspace$ to $\Rplus$ defined by
\begin{align*}
\Psi f(x) := \log \sup_{F\in\projmeasuredfoliations}
                 \frac{f(F)}{\extlength_x(F)},
\end{align*}
for all $x$ in $\teichmullerspace$.

\begin{lemma}
\label{lem:argsup_convergence}
Let $q$ be a quadratic differential with uniquely-ergodic vertical foliation
$V(q)$, and let $f:\unitlams\to\Rplus$ be a bounded function such that
$f(V(q))>0$. For each $t\in\Rplus$, let  $f(\cdot)/\extlength_{R(q;t)}(\cdot)$
attain its maximum over $\unitlams$ at $F_t$.
%Then, $F_t\to V(q)$ as $t\to\infty$.
Then, $F_t$ converges to $V(q)$ as $t$ tends to infinity.
\end{lemma}
\begin{proof}
Fix $t\in\Rplus$. We have $e^{2t}\extlength_{R(q;t)}(V(q))=1$.
Also, by Lemma~\ref{lem:finite_time_bound},
$\extlength_{R(q;t)}(F_t)\ge e^{-2t}i(F_t,V(q))^2$.
Combining these, and using the maximising property we have assumed for $F_t$,
we get
\begin{align*}
i(F_t,V(q))^2 \le \frac{f(F_t)}{e^{4t}f(V(q)}.
\end{align*}
So, as $t$ tends to infinity, $i(F_t,V(q))$ converges to zero.
Since $V(q)$ is uniquely-ergodic and all the $F_t$ are in $\unitlams$,
this implies that $F_t$ converges to $V(q)$.
\end{proof}

\begin{lemma}
\label{lem:horocompact}
The horofunction compactification is isomorphic to the compactification
$(\extfunc,\compE)$.
\end{lemma}
\begin{proof}
The continuity of $\Psi$ follows immediately from the compactness of
$\projmeasuredfoliations$ and the topology we are using on $\compE$.

For $y\in\teichmullerspace$,
\begin{align*}
\Psi\extfunc_y(x)
%\Psi\after\extfunc_y(x)
   &= \log \sup_{F\in\projmeasuredfoliations}
             \frac{\extlength_y(F)}{\extlength_x(F)}
    - \log K_y \\
   &= d(\cdot,y) - d(b,y).
\end{align*}
So, $\Psi$ is a continuous extension to $E$ of the map
$y\mapsto d(\cdot,y) - d(b,y)$.

It remains to show that $\Psi$ is injective.
Let $f$ and $g$ be distinct elements of $\compE$.
Exchanging $f$ and $g$ if necessary, we have
$f(G)<g(G)$ for some uniquely ergodic $G\in\unitlams$, since such foliations
are dense in $\unitlams$.
Since $f$ and $g$ are continuous, we may choose
a neighbourhood $N$ of $G$ in $\unitfoliations$ small enough that there are
real numbers $u$ and $v$ such that
\begin{align*}
f(F) \le u < v \le g(F),
\qquad\text{for all $F\in N$}.
\end{align*}
By Lemma~\ref{lem:argsup_convergence}, we can find a point $p$ in
$\teichmullerspace$ such that the supremum over $\unitlams$
of $f(\cdot)/\extlength_p(\cdot)$
is attained in the set $N$. Putting all this together, we have
\begin{align*}
\sup_{\unitlams}\frac{f(\cdot)}{\extlength_p(\cdot)}
  & = \sup_{N}\frac{f(\cdot)}{\extlength_p(\cdot)}
   \le \sup_{N}\frac{u}{\extlength_p(\cdot)} \\
  & < \sup_{N}\frac{v}{\extlength_p(\cdot)}
   \le \sup_{N}\frac{g(\cdot)}{\extlength_p(\cdot)}
   \le \sup_{\unitlams}\frac{g(\cdot)}{\extlength_p(\cdot)}.
\end{align*}
Thus, $\Psi f(p)<\Psi g(p)$, which implies that $\Psi f$ and $\Psi g$ differ.
We have proved that $\Psi$ is injective.

The result now follows on applying Lemma~\ref{lem:samecompactification}.
\end{proof}

\begin{reptheorem}{thm:homeo}
The Gardiner--Masur compactification and the horofunction compactification
of $\teichmullerspace$ are isomorphic.
\end{reptheorem}

\begin{proof}
This follows from Lemmas~\ref{lem:gardmasur} and~\ref{lem:horocompact}.
\end{proof}

\section{The set of Busemann points}
\label{sec:busemann}

For distinct points $x$ and $y$ in $\teichmullerspace$, let $Q(x,y)$ be the
unit-area quadratic differential at $x$ such that $R(q;\cdot)$ passes
through $y$.  

\begin{theorem}
\label{thm:busemannformula}
Let $q$ be a quadratic differential.
Then, the Teichm\"uller geodesic ray $R(q;\cdot)$ converges in the
Gardiner--Masur compactification to the point
\begin{align*}
\Big[\extfunc_q(\cdot)\Big]
%   := \left[\sum_j \frac{i(G_j,\cdot)^2}{i(G_j,\tau_x(F))} \right],
   := \Big[\sum_j \frac{i(G_j,\cdot)^2}{i(G_j,H(q))} \Big],
\end{align*}
where the $\{G_j\}$ are the indecomposable components of $V(q)$.
\end{theorem}
\begin{proof}
This is essentially a restatement of Theorem~\ref{thm:extreme_len_asymptotics}.
\end{proof}
\begin{corollary}
\label{cor:busemannformula}
Two geodesics with initial unit-area quadratic differentials $q$ and $q'$
converge to the same point of the Gardiner--Masur boundary if $q$ and $q'$
are modularly equivalent.
\end{corollary}
\begin{proof}
Since $q$ and $q'$ are modularly equivalent, $V(q) =\sum_j \alpha_j G_j$ and
$V(q') =\sum_j \alpha'_j G_j$ for some set of mutually non-intersecting
indecomposable measured foliations $\{G_j\}_j$ and positive coefficients
$\{\alpha_j\}_j$ and $\{\alpha'_j\}_j$
satisfying~(\ref{eqn:ratios}).
So $\extfunc_q = \extfunc_{q'}$. We now apply the theorem.
\end{proof}
\begin{remark}
We will prove the converse to this corollary in
Theorem~\ref{thm:modular_bijection}.
\end{remark}

\begin{definition}
Let $T$ be a sub-interval of $\Rplus$. A path $\gamma:T\to X$ in a metric
space $(X,d)$ is an \emph{optimal path} for a function $f:X\to\R$
if it is geodesic and if $f(\gamma(s))=d(\gamma(s),\gamma(t)) + f(\gamma(t))$
for all $s,t\in T$ with $s<t$.
\end{definition}

\begin{lemma}
\label{lem:geos_are_opts}
Let $(X,d)$ be a metric space, and let $\gamma:\R\to X$ be a geodesic line
converging in the forward direction to $\xi$ in the horofunction boundary
of $X$. Then, $\gamma$ is an optimal path for~$\xi$.
\end{lemma}
\begin{proof}
For all $s,t\in\R$ with $s<t$,
\begin{align*}
\xi(\gamma(s)) - \xi(\gamma(t))
   &= \lim_{u\to\infty}
         \Big( d(\gamma(s),\gamma(u)) - d(\gamma(t),\gamma(u)) \Big) \\
   &= d(\gamma(s),\gamma(t)).
\qedhere
\end{align*}
\end{proof}

Recall that a function $f:X\to\R$ is $1$-Lipschitz if $f(x)-f(y)\le d(x,y)$
for all points $x$ and $y$ in $X$.
The following lemma shows that optimal paths for $1$-Lipschitz functions
may be ``spliced'' together.
\begin{lemma}
\label{lem:splice}
Let $T_1$ and $T_2$ be two sub-intervals of $\R$ with non-empty intersection.
Let $\gamma_1:T\to X$  and $\gamma_2:T\to X$ be optimal paths for a
$1$-Lipschitz function $f:X\to\R$, such that $\gamma_1$  and $\gamma_2$ agree
on $T_1\intersection T_2$. Then, the path defined,
for $t\in T_1\union T_2$, by
\begin{align*}
\gamma(t)
   := \begin{cases}
         \gamma_1(t), & \text{if $t\in T_1$,} \\
         \gamma_2(t), & \text{if $t\in T_2$,}
      \end{cases}
\end{align*}
is an optimal path for $f$.
\end{lemma}
\begin{proof}
Swap the indices if necessary so that $T_1\backslash T_2 \subset (-\infty,t)$
and $T_2\backslash T_1 \subset (t,\infty)$,
for some $t\in T_1\intersection T_2$.
Let $t_1,t_2\in T_1\union T_2$ be such that $t_1\le t\le t_2$.
Since $\gamma_1$  and $\gamma_2$ are optimal paths for $f$,
\begin{align*}
d(\gamma(t_1),\gamma(t)) &= t - t_1 = f(\gamma(t_1)) - f(\gamma(t))
\qquad\text{and} \\
d(\gamma(t),\gamma(t_2)) &= t_2 - t = f(\gamma(t)) - f(\gamma(t_2)).
\end{align*}
Adding these equations and using the $1$-Lipschitzness of $f$ gives
\begin{align*}
d(\gamma(t_1),\gamma(t)) + d(\gamma(t),\gamma(t_2))
   = t_2 - t_1
   = f(\gamma(t_1)) - f(\gamma(t_2))
   \le d(\gamma(t_1),\gamma(t_2)).
\end{align*} 
Applying the triangle inequality, we get that these inequalities are actually
equalities.
The same equalities hold trivially when $t_1$ and $t_2$ are both less than
or both greater than $t$, since in this case they are both in $T_1$ or both
in $T_2$, respectively.
\end{proof}

For any $x\in\teichmullerspace$, and $F\in\measuredfoliations\backslash\{0\}$,
define $\tau_x(F)$ to be the unique $G\in\measuredfoliations\backslash\{0\}$
such that $F$ and $G$ are the vertical and horizontal foliations of a quadratic
differential based at $x$. In other words, $F$ and $\tau_x(F)$ together define
a singular flat metric on $S$ that is in the conformal class of metrics $x$.
By the Hubbard--Masur theorem, $\tau_x(F)$ is jointly continuous in $x$ and
$F$.

\begin{reptheorem}{thm:modular}
Every modular equivalence class of quadratic differentials has a representative
at each point of Teichm\"uller space. This representative is unique up to
multiplication by a positive constant.
\end{reptheorem}

\begin{proof}
First we prove uniqueness. Let $q$ and $q'$ be two unit-area quadratic
differentials at the same point $x$ of $\teichmullerspace$,
and suppose that $q$ and $q'$ are modularly equivalent.
Consider the geodesics $\gamma:\R\to\teichmullerspace$
and $\gamma':\R\to\teichmullerspace$ passing through $x$ at time
zero and having directions $q$ and $q'$, respectively.
By Corollary~\ref{cor:busemannformula}, these geodesics both converge in the
forward direction to the same Busemann point $\xi$.

By Lemma~\ref{lem:geos_are_opts}, both $\gamma$ and $\gamma'$ are optimal paths
for the horofunction $\xi$. So, by Lemma~\ref{lem:splice}, the path
\begin{align*}
\overline\gamma(t)
   := \begin{cases}
      \gamma'(t), & t<0, \\
      \gamma(t), & t\ge 0,
      \end{cases}
\end{align*}
is also optimal for $\xi$. In particular, $\overline\gamma$ is a geodesic.
However, Teichm\"uller
geodesics are uniquely extendable~\cite{kravetz}. We conclude that
$\gamma$ and $\gamma'$ are identical, from which it follows that
$V(q)$ and $V(q')$ are identical. This further implies that $q=q'$.

The following proof of existence uses the uniqueness and is similar to the proof
in the special case of Jenkins--Strebel foliations;
see for example~\cite[Theorem 3]{hubbard_masur_quadratic}.

We use induction on the number $J$ of indecomposable components comprising
each member of the given modular equivalence class.

When $J=1$, there exists by the Hubbard--Masur theorem a quadratic differential
at $x$ whose vertical foliation is proportional to the single component
of the modular equivalence class. In this case~(\ref{eqn:ratios}) is trivially
satisfied.

Assume the result is true when the number of indecomposable components is less
than $J$. Suppose we are given a modular equivalence class whose members
have $J$ indecomposable components proportional to $\{G_j\}_{1\le j\le J}$.
For each $(\lambda_j)_j$ in $(0,\infty)^J$, define the
measured foliation class $V_\lambda:=\sum_{j=1}^J \lambda_j G_j$.
Consider the map $M$ from $(0,\infty)^J$ to itself given by
\begin{align}
\label{map_foliations_moduli}
(\lambda_j)_j \mapsto \Big(\frac{\lambda_j}{i(G_j,\tau_x(V_\lambda))}\Big)_j.
\end{align}

By the theorem of Hubbard--Masur, $\tau_x$ is a continuous function.
Also, for any $j$, since $i(G_j,V_\lambda)=0$,
we have $i(G_j,\tau_x(V_\lambda))>0$.
It follows that $M$ can be extended continuously to
$\Rplus^J \backslash\{0\}$.

Observe that $M$ satisfies $M(\alpha \lambda)=M(\lambda)$, for all $\alpha>0$
and vectors $\lambda = (\lambda_j)_j$. So $M$ induces a continuous self map
$\tilde M$ of the projective space $\proj (\Rplus^J \backslash\{0\})$.
The uniqueness proved above is precisely that this map is injective.
The space $\proj (\Rplus^J \backslash\{0\})$ has the structure of a
closed simplex, and $\tilde M$ leaves each open face invariant.
By the induction hypothesis, $\tilde M$ is a surjection on each open face.
We conclude that $\tilde M$ is a homeomorphism on the boundary of
$\proj (\Rplus^J \backslash\{0\})$.
But $\proj (\Rplus^J \backslash\{0\})$ has the topology of a closed disk,
and every injective map of a closed disk that is a homeomorphism on the
boundary is a homeomorphism. Therefore, $M$ is surjective.
\end{proof}

For each $x\in\teichmullerspace$ and $G\in\measuredfoliations\backslash\{0\}$,
define $q(x,G)$ to be the quadratic differential at $x$ with vertical
foliation $G$.

\begin{lemma}
\label{lem:simplex_closed}
Let $\{G_j\}_j$ be a set of mutually non-intersecting indecomposable measured
foliations, and define the set of measured foliations
\begin{align*}
\Delta := \Big\{ \sum_j \lambda_j G_j
                  \mid \text{$\lambda_j \ge 0$ for all $j$} \Big\}
               \backslash\{0\}.
\end{align*}
Then, the set
$\{ [\extfunc_{q(x,G)}]
   \mid \text{$x\in\teichmullerspace$ and $G\in \Delta$} \}$
is a closed subset of the Gardiner--Masur boundary.
\end{lemma}
\begin{proof}
Combining Theorem~\ref{thm:modular} and Corollary~\ref{cor:busemannformula},
we see that
\begin{align*}
\{ [\extfunc_{q(x,G)}] \mid \text{$x\in\teichmullerspace$ and $G\in \Delta$} \}
   = \{ [\extfunc_{q(b,G)}] \mid \text{$G\in \Delta$} \} =: D.
\end{align*}
It follows easily from the arguments in the second part of the proof of
Lemma~\ref{thm:modular} that the map from $\proj(\Rplus^J\backslash\{0\})$
to the Gardiner--Masur boundary given by
\begin{align*}
(\lambda_j)_j
   \mapsto \extfunc_{q(b,V_\lambda)}(\cdot)
   =  \Big( \sum_j
        \frac{\lambda_j i(G_j,\cdot)^2}{i(G_j, \tau_b(V_\lambda))}
        \Big)^{1/2}
\end{align*}
is continuous. Since the domain is compact, the image is compact.
\end{proof}

A \emph{min-plus measure} is a lower semicontinuous function from some
set to $\R\union\{\infty\}$.

\begin{theorem}
\label{thm:busemann_functions}
A horofunction is a Busemann point if and only if it can be expressed
$\Psi f$ for some function $f$ in the set
$\{\extfunc_q \mid \text{$q$ is a quadratic differential}\}$.
\end{theorem}

\begin{proof}
That horofunctions of the above form are Busemann points follows from
Theorem~\ref{thm:busemannformula}.

Let $y_n$ be a sequence in $\teichmullerspace$ converging to a Busemann point
$\eta$. So, $h_n:=d(\cdot,y_n)-d(b,y_n)$ converges to $\eta$ uniformly on
compact sets.

Let $x\in\teichmullerspace$.
By taking a subsequence if necessary, we may assume that the sequence
$q_n:=Q(x,y_n)$ converges to a unit-area quadratic differential $q^{x}$.
For $n\in\N$ and $t\in\Rplus$, let $\gamma(t):=R(q^{x};t)$ and
$\gamma_n(t):=R(q_n;t)$. Observe that by Theorems~\ref{thm:homeo}
and~\ref{thm:busemannformula}, the geodesic ray $\gamma$ converges to the
Busemann point $\xi^x:=\Psi\extfunc_{q^x}$.

The continuity of $R(\cdot;\cdot)$ gives, for all $t\in\Rplus$,
that $\gamma_n(t)$ converges to $\gamma(t)$ as $n$ tends to infinity.
We conclude that $h_n(\gamma_n(t))$ converges to $\eta(\gamma(t))$ for all $t$.
For each $n\in\N$, since $\gamma_n$ is a geodesic passing through $y_n$,
we have $h_n(x)=t+ h_n(\gamma_n(t))$ for all $t\le d(x,y_n)$.
Taking the limit in $n$,
we get $\eta(x) = t + \eta(\gamma(t))$ for all $t\in\Rplus$.
Now taking the limit as $t$ tends to infinity, we get,
by Proposition~\ref{prop:lim_along_geos},
\begin{align}
\label{eqn:equality_at_x}
\eta(x) = \xi_x(x) + H(\xi_x,\eta),
\end{align}
since $\xi_x$ is the Busemann point to which $\gamma$ converges.
This is true true for all $x\in\teichmullerspace$.

We now allow $x$ to vary.
%By~\cite[Lemma 3.6]{AGW-m},
By~\cite[Lemma 5.1]{walsh_stretch},
$\eta(\cdot) \le \xi(\cdot) + H(\xi,\eta)$ for each horofunction $\xi$. So,
\begin{align}
\label{eqn:minplus_eta}
\eta(\cdot):=\inf_{x\in\teichmullerspace}
                \Big(\xi_x(\cdot) + H(\xi_x,\eta)\Big).
\end{align}

It follows from~\cite[Proposition~5.1]{miyachi_teichmuller} that
$i(V(q^x),V(q^y))=0$ for all $x$ and $y$ in $\teichmullerspace$.
Therefore, there exists a finite set $\{G_j\}_j$ of mutually non-intersecting
indecomposable measured foliations such that, for all $x\in\teichmullerspace$,
the foliation $V(q^x)$ is in the set
\begin{align*}
\Delta := \Big\{ \sum_j \lambda_j G_j
                  \mid \text{$\lambda_j \ge 0$ for all $j$} \Big\}.
\end{align*}
By Lemma~\ref{lem:simplex_closed}, the set
$D := \{ \extfunc_{q(x,G)}
             \mid \text{$x\in\teichmullerspace$ and $G\in \Delta$} \}$
is a closed subset of the horofunction boundary.
Obviously, $\xi_x$ is in $D$, for each $x\in\teichmullerspace$.

From~(\ref{eqn:minplus_eta}), we may write 
\begin{align*}
\eta(\cdot):=\inf_{\xi\in B}\Big( \xi(\cdot) + \nu(\xi)\Big),
\end{align*}
where $B$ is the set of Busemann points and $\nu$ is a min-plus measure
on $B$ taking the value $\infty$ outside $D$.
Since $\eta$ is a Busemann point it may be written
$\eta=\inf_{\xi\in B}( \xi + \nu'(\xi) )$ where $\nu'$ takes the value $0$ at
$\eta$, and the value $\infty$ everywhere else.
By~\cite[Theorem 1.2]{walsh_minimum}, there is a min-plus measure $\rho$
on $B$ satisfying $\eta=\inf_{\xi\in B}( \xi + \rho(\xi))$ that is greater
than or equal to both $\nu$ and $\nu'$. Since $\eta$ is not identically
$\infty$, neither is $\rho$, and therefore $\eta$ must be in $D$.
We have thus proved that $\eta$ is of the required form.
\end{proof}

\begin{reptheorem}{thm:busemann_bijection}
Let $p$ be a point of $\teichmullerspace$ and $\xi$ be a Busemann point
of the horoboundary. Then, there exists a unique geodesic ray starting at $p$
and converging to $\xi$.
\end{reptheorem}

\begin{proof}
By Theorem~\ref{thm:busemann_functions}, $\xi=\extfunc_{q}$ for some modular
equivalence class $[q]$ of quadratic differentials.
By Theorem~\ref{thm:modular}, this modular equivalence class has a
representative $q$ at $p$.
By Theorem~\ref{thm:busemannformula}, the geodesic $R(q;\cdot)$
converges to~$\xi$. This geodesic starts at $p$.

Suppose that $\gamma$ and $\gamma'$ are two geodesics starting at $p$ and
converging to $\xi$. Using the same reasoning as in the uniqueness part of
the proof of Theorem~\ref{thm:modular}, one can show that
$\gamma$ and $\gamma'$ are identical.
\end{proof}

\begin{reptheorem}{thm:modular_bijection}
Two Busemann points $\extfunc_{q}$ and $\extfunc_{q'}$ are identical
if and only if $q$ and $q'$ are modularly equivalent.
\end{reptheorem}

\begin{proof}
It was proved in Corollary~\ref{cor:busemannformula}
that $\extfunc_{q}$ and $\extfunc_{q'}$ are identical
when $q$ and $q'$ are modularly equivalent.

%Let $q=(x,V(q))$ and $q'=(x',H_{q'})$ be two quadratic differentials that are
Let $q$ and $q'$ be quadratic differentials based at points $x$ and $y$,
respectively, that are not modularly equivalent.
By Theorem~\ref{thm:modular}, we can find a
quadratic differential $\tilde q$ at $x$ that is modularly equivalent
to $q'$, and hence different from $q$.
So, $q$ and $q'$ define different geodesics emanating from $p$,
and, by Theorem~\ref{thm:busemann_bijection}, the two geodesics have different
limits.
We conclude that $\extfunc_{q} \neq \extfunc_{\tilde q} = \extfunc_{q'}$.
\end{proof}

\begin{lemma}
\label{lem:proportional}
Let $q$ be a quadratic differential.
If $V(q) = \sum_j G_j$ is written as a sum of indecomposable measured
foliations, possibly scalar multiples of one another, then
\begin{align*}
\extfunc_{q}^2(F) = \sum_j \frac{i(G_j,F)^2}{i(G_j,H(q))},
\qquad\text{for all $F\in\measuredfoliations$}.
\end{align*}
\end{lemma}
\begin{proof}
Let $F'$ be some indecomposable component of $F$, and let $J'$ be the set of
indices $j$ for which $G_j=\lambda_j F'$ for some $\lambda_j>0$.
Clearly, $\sum_{j\in J'} \lambda_j = 1$.
So,
\begin{align*}
\sum_{j\in J'} \frac{i(G_j,F)^2}{i(G_j,H(q))}
%   = \sum_{j\in J'} \frac{i(F,\lambda^j F')^2}{i(H(q),\lambda^j F')}
   = \sum_{j\in J'} \frac{\lambda_j i(F',F)^2}{i(F',H(q))}
   = \frac{i(F',F)^2}{i(F',H(q))}.
\end{align*}
Since this is true for every indecomposable component $F'$ of $F$,
the result follows.
\end{proof}

\begin{lemma}
\label{lem:smaller}
Let $G = \sum_j^J G_j$ be written as a sum of measured
foliations, possibly scalar multiples of one another,
and let $H\in\measuredfoliations$ be such that $i(H,G_j)>0$ for all~$j$.
Then
\begin{align*}
\frac{i(G,F)^2}{i(G,H)} \le \sum_j \frac{i(G_j,F)^2}{i(G_j,H)},
\qquad\text{for all $F\in\measuredfoliations$.}
\end{align*}
If the $G_j$ are not all scalar multiples of the same measured foliation,
then the inequality is strict for some $F\in\measuredfoliations$.
\end{lemma}
\begin{proof}
Observe first that,
for all $g_1, g_2 \in [0,\infty)$ and $h_1, h_2 \in (0,\infty)$,
\begin{align*}
\frac{(g_1+g_2)^2}{h_1+h_2}
   \le \frac{g_1^2}{h_1} + \frac{g_2^2}{h_2},
%\qquad\text{for all $g_1, g_2 \in [0,\infty)$ and $h_1, h_2 \in (0,\infty)$}
\end{align*}
and that equality occurs precisely when $g_1/h_1 = g_2/h_2$.

We use induction on $J$. The lemma is trivially true when $J=1$.

Assume that it is true when there are $J-1$ terms in the sum.
Write $G = G' + G_J$, where $G' := \sum_{j=1}^{J-1} G_j$ is the sum of
the first $J-1$ terms.
Using the inequality above and then the induction hypothesis, we get,
for all $F\in\measuredfoliations$,
\begin{align*}
\frac{i(G,F)^2}{i(G,H)}
   \le \frac{i(G',F)^2}{i(G',H)} + \frac{i(G_J,F)^2}{i(G_J,H)}
   \le \sum_{j=1}^{J-1} \frac{i(G_j,F)^2}{i(G_j,H)}
           + \frac{i(G_J,F)^2}{i(G_J,H)}.
\end{align*}
Thus the inequality holds when there are $J$ terms.

Equality for $F\in\measuredfoliations$ is equivalent to
\begin{align*}
\frac{i(G',F)}{i(G',H)} = \frac{i(G_J,F)}{i(G_J,H)}.
\end{align*}
If this is true for all $F$, then $G_J$ is projectively equivalent to
$G'$, and hence to $G$. Since the ordering of the sum is arbitrary,
the same applies to each term.
\end{proof}

\begin{lemma}
\label{lem:smallmodulus}
Let $G_n$ be a sequence in $\measuredfoliations$ converging to a non-zero
element $G$ of $\measuredfoliations$, and let $H_n$ be a sequence in
$\measuredfoliations$ such that $H_n$ is proportional to an indecomposable
component of $G_n$ for all $n$, and $H_n$ converges to $0$ as $n$ tends to
infinity. Then, for all $x\in\teichmullerspace$ and $F\in\measuredfoliations$,
\begin{align*}
\lim_{n\to\infty}\frac{i(H_n,F)^2}{i(H_n,\tau_x(G_n))} = 0.
\end{align*}
\end{lemma}
\begin{proof}
Let $\lambda_n$ be a sequence of positive real numbers such that
the quadratic differential $q(x,\lambda_n H_n)$ has unit area, for all $n$.
Since the set of unit-area quadratic differentials at $x$ is compact,
by taking a subsequence if necessary, we may assume that $\lambda_n H_n$
converges to an element $H$ of $\measuredfoliations\backslash\{0\}$.
For any $F\in\measuredfoliations$, we have that $i(H_n,F)$ converges to $0$,
and $i(\lambda_n H_n,F)$ converges to $i(H,F)$.
Also, $i(\lambda_n H_n, \tau_x(G_n))$ converges to $i(H,\tau_x(G))$.
Since $i(\lambda_n H_n,G_n)=0$ for all $n$, we have $i(H,G)=0$, which
implies that $i(H,\tau_x(G))>0$.
Therefore,
\begin{align*}
\lim_{n\to\infty}\frac{i(H_n,F)^2}{i(H_n,\tau_x(G_n))}
   = \lim_{n\to\infty}\frac{i(H_n,F)i(\lambda_n H_n,F)}
                           {i(\lambda_n H_n,\tau_x(G_n))}
   = 0.
\tag*{\qedhere}
\end{align*}
\end{proof}

\begin{lemma}
\label{lem:unique_optimal}
Let $\gamma:\Rplus\to\teichmullerspace$ be a geodesic ray starting from a
point $\gamma(0)=p$ and converging to a Busemann point $\xi$.
For any $r\ge 0$, the point $\gamma(r)$ is the unique point $x$ satisfying
$d(p,x) = \xi(p) - \xi(x) = r$.
%use d for Teichmuller distance?
\end{lemma}
\begin{proof}
By Lemma~\ref{lem:geos_are_opts}, $\gamma(r)$ satisfies this condition.

Let $x$ be any point of $\teichmullerspace$ satisfying the condition.
By Lemma~\ref{thm:busemann_bijection},  there exists a geodesic ray
$\gamma':\Rplus\to\teichmullerspace$ starting at $x$ and converging to $\xi$,
and, by Lemma~\ref{lem:geos_are_opts}, this ray is an optimal path for $\xi$.

Let $\gamma'':[0,r]\to\teichmullerspace$ be the geodesic segment connecting $p$
and $x$. Since $\xi$ is $1$-Lipschitz,
\begin{align*}
   \xi(\gamma''(t)) - \xi(x) &\le r-t
\qquad\text{and} \\
   \xi(p) - \xi(\gamma''(t)) &\le t,
\end{align*}
for all $t\in[0,r]$.
Combining this with the assumption on $x$, we get $\xi(p)-\xi(\gamma''(t)) = t$,
for all $t\in[0,r]$. It follows that $\gamma''$ is an optimal path for $\xi$.
Applying Lemma~\ref{lem:splice}, we see that the path
\begin{align*}
\gamma'''(t)
   := \begin{cases}
         \gamma''(t), & \text{if $t\in[0,r]$,} \\
         \gamma'(t), & \text{if $t\ge r$,}
      \end{cases}
\end{align*}
is an optimal path for $\xi$, and hence a geodesic.
But, by Lemma~\ref{thm:busemann_bijection}, there is only one geodesic starting
at $p$ and converging to $\xi$. Therefore, $\gamma'''$ is identical to $\gamma$,
and $\gamma(r)=x$.
\end{proof}

\begin{theorem}
\label{thm:convergence_criterion}
Let $q_n$ be a sequence of unit-area quadratic differentials
based at $b\in\teichmullerspace$.
Then, $\extfunc_{q_n}$ converges to a Busemann
point $\extfunc_{q}$ if and only if both the following hold:
\begin{itemize}
\renewcommand{\labelitemi}{(i)}
\item
$q_n$ converges to $q$;
\renewcommand{\labelitemi}{(ii)}
\item
%let $G_n$ be any sequence of indecomposable elements of $\measuredfoliations$
%such that, for each $n\in\N$, we have $G_n$ is a component of $F_n$.
%Then, every limit point of $G_n$ is indecomposable.
for every sequence $(G^n)_n$ of indecomposable elements of $\measuredfoliations$
such that, for each $n\in\N$, $G^n$ is a component of $V(q_n)$, we have
that every limit point of $G^n$ is indecomposable.
\end{itemize}
\end{theorem}
\begin{proof}
Assume conditions (i) and (ii) hold.
We wish to show that $\extfunc_{q_n}$ converges to $\extfunc_{q}$ in the
Gardiner--Masur compactification. So, consider any limit point of this
sequence. By taking a subsequence if necessary, we can assume that
$\extfunc_{q_n}$ actually converges to this point.

For each $n\in\N$, we can write $V(q_n)=\sum_{j=1}^J G_j^n$,
with an upper bound on $J$ depending on the topology of the surface.
By taking a subsequence if necessary, we can ensure that $G_j^n$ converges
to some $G_j$ in $\measuredfoliations$ for each $j$.
By hypothesis, $V(q)=\sum_j G_j$, and each $G_j$ is indecomposable.
Note that this is not necessarily a decomposition of $V(q)$ into indecomposable
components since some of the $G_j$ may be scalar multiples of each other.

The convergence of $q_n$ implies that $H(q_n)$ converges to $H(q)$.
We deduce that $a_j^{n}:= i(G_j^n,H(q_n))$ converges to
$a_j:= i(G_j,H(q))$, for each $j$.

Let $F\in\measuredfoliations$.
For each $j$ such that $a_j$ is zero, we have that
$G_j^n$ converges to zero, and hence, by Lemma~\ref{lem:smallmodulus}, that
$i(G_j^n,F)^2 / a_j^{n}$  converges to zero.
For all other $j$, we have that $i(G_j^n,F)^2 / a_j^{n}$ converges to
$i(G_j,F)^2/a_j$.
It follows that $\extfunc_{q_n}(F)$ converges to $\extfunc_{q}(F)$,
by Lemma~\ref{lem:proportional}.

Now assume that $\extfunc_{q_n}$ converges to $\extfunc_{q}$.
So, the associated horofunctions $\xi_n:=\Psi\extfunc_{q_n}$ converge uniformly
on compact sets to $\xi:=\Psi\extfunc_{q}$.
For each $n\in\N$, let $z_n:=R(q_n;1)$.
Observe that $\teichdist(b,z_n)=1$ and $\xi_n(z_n) = -1$
for all $n$. So, for any limit point $z$ of the sequence $(z_n)_n$,
we have $\teichdist(b,z)=1$ and $\xi(z) = -1$.
But, by Lemma~\ref{lem:unique_optimal}, $R(q;1)$ is the only point
of Teichm\"uller space with these properties. We conclude that $z_n$ converges
to  $R(q;1)$. It follows that $q_n$ converges to $q$, and hence
that~(i) holds.

Let $G^n$ be a sequence as in~(ii).
We may, for each $n\in\N$, write $V(q_n) = \sum_{j=0}^J G_j^n$, where $J$ is
independent of $n$, each $G_j^n$ is either zero or an indecomposable component
of $V(q)$, and $G^n=G^n_0$.

We wish to show every limit point of $(G_0^n)_n$ is indecomposable.
By taking a subsequence if necessary, we may assume that, for each $j$,
the sequence $(G_j^n)_n$ converges to some element $G_j$ of
$\measuredfoliations$. Since $q_n$ converges to $q$,
we have $V(q) = \sum_{j=0}^J G_j$. Write $a_j:= i(G_j, H(q))$, for each $j$.
As before, for any $F\in\measuredfoliations$,
\begin{align*}
\lim_{n\to\infty} \extfunc_{q_n}^2(F) = \sum_{j} \frac{i(G_j,F)^2}{a_j},
\end{align*}
where the sum is over all $j\in\{0,\dots,J\}$ such that $G_j$ is not zero.

For each $j$, we can write $G_j = \sum_{l=0}^{L_j} G_j^l$ as a sum of
projectively-distinct indecomposable measured foliations,
where $L_j$ is bounded depending on the topology of the surface.
Even though the $\{G_j^l\}_{j,l}$ are not necessarily projectively distinct,
we have, by Lemma~\ref{lem:proportional}, that
\begin{align*}
\extfunc_{q}^2(F) =  \sum_{j=0}^J \sum_{l=0}^{L_j}
                        \frac{i(G_j^l,F)^2}{i(G_j^l,H(q))}.
\end{align*}
By Lemma~\ref{lem:smaller}, for each $j$,
\begin{align}
\label{eqn:smaller_eqn}
\frac{i(G_j,F)^2}{i(G_j,H(q))}
   \le \sum_{l=0}^{L_j} \frac{i(G_j^l,F)^2}{i(G_j^l,H(q))}.
\end{align}
Since $\extfunc_{q_n}$ converges to $\extfunc_{q}$,
equality holds in~(\ref{eqn:smaller_eqn}) for all $F\in\measuredfoliations$,
and for all~$j$. Therefore, according to Lemma~\ref{lem:smaller},
for each $j$, the $\{G_j^l\}_l$ are all projectively equivalent
to $G_j$, that is, $G_j$ is indecomposable.
\end{proof}

\begin{reptheorem}{thm:convergent_rays}
Each Teichm\"uller ray $R(q;\cdot)$ is convergent to the ray $R(q';\cdot)$,
where $q'$ is the unique unit-area quadratic differential at the basepoint
that is modularly equivalent to $q$.
\end{reptheorem}

\begin{proof}
The existence and uniqueness of $q'$ was proved in Theorem~\ref{thm:modular}.
By Theorem~\ref{thm:modular_bijection}, $\extfunc_q = \extfunc_{q'}$.

By Theorem~\ref{thm:busemannformula}, $R(q;\cdot)$ converges in the
Gardiner--Masur compactification to $\extfunc_q$. But this compactification
is the same as the horocompactification by Theorem~\ref{thm:homeo}, and so
$\Psi \extfunc_{R(q;\cdot)} = d(\cdot,R(q;t)) - d(b,R(q;t))$ converges
uniformly on compact sets to $\Psi\extfunc_q=\Psi\extfunc_{q'}$,
as $t$ tends to infinity.
Choose $s\in\Rplus$. For each $t$, let $z(t):= R(q(t);s)$,
where $q(t):=Q(b,R(q;t))$ is the initial quadratic differential of the
Teichm\"uller geodesic segment from $b$ to $R(q;t)$.
We have $d(b,z(t))=s$ and $\Psi \extfunc_{R(q;t)} = -s$, for all $t$.
Therefore, any limit point $z$ of $z(t)$ satisfies $d(b,z)=s$ and
$\Psi\extfunc_{q'}(z)=-s$, and so, by Lemma~\ref{lem:unique_optimal},
$z=R(q';s)$. We deduce that $z(t)$ converges to $R(q';s)$ as $t$ tends to
infinity. The conclusion now follows, since $s$ was chosen arbitrarily.
\end{proof}

\section{The detour metric on the boundary}
\label{sec:detour}

In this section we calculate the detour cost and detour metric of the
Teichm\"uller metric. The technique will be similar to that used
in~\cite{walsh_stretch} to calculate the same quantities for Thurston's
Lipschitz metric.

Let $G'\in\measuredfoliations$ be expressed as $G'=\sum_j G_j$
in terms of its indecomposable elements. For $G\in\measuredfoliations$,
we write $G\ll G'$ if $G$ can be expressed as $G=\sum_j \lambda_j G_j$,
where each coefficient $\lambda_j$ is a non-negative number.

\begin{lemma}
\label{lem:cross_one}
Let $F_j; j\in\{0,\dots,J\}$ be a finite set of mutually non-intersecting
indecomposable non-zero measured foliations such that no two are projectively
equivalent, and let $C>0$.
Then, there exists a curve class $\alpha\in\curves$ such that
$i(F_0,\alpha) > C i(F_j,\alpha)$ for all $j\in J\backslash\{0\}$.
\end{lemma}
\begin{proof}
This is a restatement of~\cite[Lemma~6.3]{walsh_stretch}.
\end{proof}

\newcommand\horiz{H}
\newcommand\verti{V}

\begin{lemma}
\label{lem:max_intersection_ratio}
Let $q$ and $q'$ be quadratic differentials at $b$.
If $\verti(q)\ll\verti({q'})$, then
\begin{align*}
%\sup_{\substack{{\eta\in\unitlams}\\{i(\mu,\eta)\neq 0}}}
%     \frac{i(\mu,\eta)}{i(\nu,\eta)}
\sup \Big\{
     \frac{\extfunc_{q}^2(F)}{\extfunc_{q'}^2(F)}
   \mid \text{$F\in\unitlams$}
\Big\}
   = \max_j \frac{\lambda_j i(G_j, \horiz(q'))}
                 {i(G_j, \horiz(q))},
\end{align*}
where $\verti(q)$ is expressed as $\verti(q) = \sum_j \lambda_j G_j$
in terms of the indecomposable components $G_j$ of $\verti({q'})$.
If $\verti(q)\not\ll\verti({q'})$, then the supremum is $+\infty$.
\end{lemma}
\begin{remark}
Here, and in similar situations, we interpret the supremum to be over
the set where the ratio is well defined, that is, excluding values of $F$
for which both the numerator and the denominator are zero.
\end{remark}
\begin{proof}
If $i(\verti(q),\verti({q'}))>0$, then we take $F:=\verti({q'})$, so that
the supremum is infinity.

So, assume that $i(\verti(q),\verti({q'}))=0$.
So we can write $\verti(q)=\sum_j g_j G_j$ and $\verti({q'})=\sum_j g'_j G_j$,
where the $G_j; j\in\{0,\dots,J\}$ are are finite set of mutually
non-intersecting indecomposable measured foliations, and the $\{g_j\}$
and $\{g'_j\}$ are non-negative coefficients such that, for all $j$,
either $g_j$ or $g'_j$ is positive.

Let $\iota_j:= i(G_j, \horiz(q))$ and $\iota'_j:= i(G_j, \horiz(q'))$,
for all $j$.

Relabel the indices so that the $j$ for which $g_j \iota'_j / g'_j \iota_j$
is the largest is $j=0$.
So, $g_j g'_0 / \iota_j \iota'_0 \le g_0 g'_j / \iota_0 \iota'_j$, for all $j$.
%Multiplying by $i(F,G_j)$, summing over $j$, and rearranging, we get
Therefore, for all $F\in\measuredfoliations$,
\begin{align*}
\frac{g'_0}{\iota'_0} \sum_j g_j i(F, G_j)^2/\iota_j
   &\le \frac{g_0}{\iota_0} \sum_j g'_j i(F, G_j)^2/\iota'_j,
\end{align*}
and hence
\begin{align*}
E(F) &:= \frac{\extfunc_{q}^2(F)}{\extfunc_{q'}^2(F)}
%    = \frac{\sum_j i^2(F,g_j G_j)/a_j}{\sum_j i^2(F, g'_j G_j)/a'_j} \\
      = \frac{\sum_j g_j i(F, G_j)^2/\iota_j}{\sum_j g'_j i(F, G_j)^2/\iota'_j} 
     \le \frac{g_0 \iota'_0}{\iota_0 g'_0}.
\end{align*}

For any $C>0$, we may apply Lemma~\ref{lem:cross_one}
to get a measured foliation $F_C\in\measuredfoliations$
such that $i(F_C, G_0) > C i(F_C, G_j)$ for all $j\in\{1,\dots,J\}$.
By choosing $C$ large enough, we can make $E(F_C)$
as close as we like to $g_0 \iota'_0 / g'_0 \iota_0$.

We conclude that $\sup_F E(F)=g_0 \iota'_0 / g'_0 \iota_0$.
The result follows.
\end{proof}

\begin{theorem}
\label{thm:detourcost}
Let $q$ and $q'$ be unit area quadratic differentials at $b$.
If $\verti(q)\ll\verti({q'})$, then
\begin{align*}
H(\extfunc_{q'},\extfunc_{q}) =
      \frac{1}{2}\log \sup_{F\in\projmeasuredfoliations}
              \frac{\extfunc_{q'}^2(F)}{\extlength_b(F)}
     + \frac{1}{2}\log\max_j\Big(
            \frac{\lambda_j^2 i(G_j, \horiz(q'))}
                 {i(\lambda_j G_j, \horiz(q))}   \Big)
     - \frac{1}{2}\log\sup_{F\in\projmeasuredfoliations}
              \frac{\extfunc_{q}^2(F)}{\extlength_b(F)},
\end{align*}
where $\verti(q)$ is expressed as $\verti(q) = \sum_j \lambda_j G_j$
in terms of the indecomposable components $G_j$ of $\verti({q'})$.
If $\verti(q)\not\ll\verti({q'})$, then $H(\extfunc_{q'},\extfunc_q)=+\infty$.
\end{theorem}
\begin{proof}
Let $\gamma:=R({q'};\cdot)$ be the geodesic starting at $b\in\teichmullerspace$
and having initial quadratic differential ${q'}$.
By Theorem~\ref{thm:busemannformula}, $\gamma$
converges to the Busemann point $\extfunc_{q'}$.
Therefore, by Proposition~\ref{prop:lim_along_geos},
\begin{align*}
H(\extfunc_{q'},\extfunc_{q})
   &= \lim_{t\to\infty}
             \Big(\dist(b,\geo(t)) + \Psi\extfunc_{q'}(\geo(t))\Big) \\
   &= \frac{1}{2}\lim_{t\to\infty} \Big(\log \sup_{F\in\projmeasuredfoliations}
             \frac{\extlength_{\geo(t)}(F)}{\extlength_b(F)}
    + \log \sup_{F\in\projmeasuredfoliations}
             \frac{\extfunc_{q}^2(F)}{\extlength_{\geo(t)}(F)} \Big)
    - \frac{1}{2}\log \sup_{F\in\projmeasuredfoliations}
             \frac{\extfunc_{q}^2(F)}{\extlength_b(F)}.
\end{align*}
Combining Lemma~\ref{lem:gardmasur} and
Theorem~\ref{thm:extreme_len_asymptotics},
we get that $e^{-2t}\extlength_{\gamma(t)}(\cdot)$
converges uniformly on compact sets to $\extfunc_{q'}^2(\cdot)$. Therefore,
\begin{align*}
\lim_{t\to\infty} \sup_{F\in\projmeasuredfoliations}
           \frac{\extlength_{\geo(t)}(F)}{e^{2t}\extlength_b(F)}
   = \sup_{F\in\projmeasuredfoliations}
           \frac{\extfunc_{{q'}}^2(F)}{\extlength_b(F)}.
\end{align*}
From Lemma~\ref{lem:finite_time_bound}, we get
\begin{align*}
\sup_{F\in\projmeasuredfoliations}
        \frac{e^{2t}\extfunc_{q}^2(F)}{\extlength_{\geo(t)}(F)}
   \le \sup_{F\in\projmeasuredfoliations}
        \frac{\extfunc_{q}^2(F)}{\extfunc_{{q'}}^2(F)},
\qquad\text{for all $t$}.
\end{align*}
But the limit of a supremum is trivially greater than or equal to the supremum
of the limits. We conclude that
\begin{align*}
\lim_{t\to\infty} \sup_{F\in\projmeasuredfoliations}
         \frac{e^{2t}\extfunc_{q}^2(F)}{\extlength_{\geo(t)}(F)}
   = \sup_{F\in\projmeasuredfoliations}
         \frac{\extfunc_{q}^2(F)}{\extfunc_{q'}^2(F)}.
\end{align*}
The result now follows on applying Lemma~\ref{lem:max_intersection_ratio}.
\end{proof}

\begin{corollary}
\label{cor:detour_metric}
Let $q$ and $q'$ be unit area quadratic differentials at $b$. If 
%$\verti_q$ and $\verti_{q'}$ can be written in the form
$\verti(q) = \sum_j g_j G_j$ and $\verti({q'}) = \sum_j g'_j G_j$,
where $\{G_j\}$ is a finite set of mutually non-intersecting indecomposable
measured foliations, and the $g_j$ and $g'_j$ are positive coefficients,
then the detour metric between $\extfunc_q$ and $\extfunc_{q'}$ is
\begin{align*}
\delta(\extfunc_{q'},\extfunc_q)
   = \frac{1}{2}\log \max_j
         \frac{g_i i(G_j,\horiz(q'))}
              {g'_i i(G_j,\horiz(q))}
   + \frac{1}{2}\log \max_j
         \frac{g'_i i(G_j,\horiz(q))}
              {g_i i(G_j,\horiz(q'))}.
\end{align*}
If $\verti(q')$ and $\verti({q})$ can not be simultaneously written in this
form, then $\delta(\extfunc_q,\extfunc_q')=+\infty$.
\end{corollary}
%\begin{proof}
%\end{proof}
%Let $q$ and $q'$ be unit area quadratic differentials at $b$.

%\begin{lemma}
%\label{}
%\end{lemma}
%\begin{proof}
%\end{proof}

%\begin{theorem}
%\label{}
%\end{theorem}
%\begin{proof}
%\end{proof}

\bibliographystyle{plain}
\bibliography{gardinermasur}

\end{document}